\documentclass[10pt,a4paper]{article}

\usepackage{amsmath,amssymb}
\usepackage{amsthm}
\usepackage{graphicx}
\usepackage{siunitx}
\usepackage{enumitem}
\usepackage{pb-diagram}
\usepackage{tikz}
\usepackage{bm}
\usepackage{color}
\usepackage{float}
\usepackage{url}
\usepackage{comment}
\usepackage{multirow}
\usepackage{mathrsfs}
\usepackage{ascmac}
\usepackage{caption}
\usepackage[margin=25truemm]{geometry}

\theoremstyle{definition}

\newtheorem{thm}{Theorem}[section]
\newtheorem{prop}[thm]{Proposition}
\newtheorem{lem}[thm]{Lemma}
\newtheorem{cor}[thm]{Corollary}
\newtheorem{rem}{Remark}[section]
\newtheorem*{rem*}{Remark}

\numberwithin{equation}{section}

\newcommand{\BN}{\mathbf{N}}
\newcommand{\BZ}{\mathbf{Z}}

\newcommand{\BR}{\mathbf{R}}

\newcommand{\CB}{\mathcal{B}}
\newcommand{\CC}{\mathcal{C}}

\newcommand{\CH}{\mathcal{H}}
\newcommand{\CL}{\mathcal{L}}
\newcommand{\CM}{\mathcal{M}}
\newcommand{\CN}{\mathcal{N}}

\newcommand{\CP}{\mathcal{P}}
\newcommand{\CQ}{\mathcal{Q}}
\newcommand{\CX}{\mathcal{X}}
\newcommand{\CY}{\mathcal{Y}}

\title{Limit distributions of the threshold radius for the maximum degree and the associated point configurations in random geometric graphs}
\author{Junpei Otsuka
\thanks{Master's Program in Mathematics, Graduate School of Science and Technology, University of Tsukuba, Ibaraki, 305-8571, Japan.
E-mail: s2520128@u.tsukuba.ac.jp}}
\date{\today}
\begin{document}
\maketitle

\begin{abstract}
    A random geometric graph $G(\CX_n, r_n)$ is formed by taking a binomial process $\CX_n$ as the set of vertices and joining any two distinct points with an edge if they lie within distance $r_n$ of each other.
    We investigate the limit distribution of the threshold radius for which the maximum degree of the graph is at least a given value that depends on $n$.
    In addition, given the radii $(r_n)_{n \in \BN}$, we examine the limiting behavior of the point process formed by the vertices that achieve the maximum degree.
    Roughly speaking, the limiting process exhibits a compound Poisson behavior in the regime where the maximum degree remains bounded, due to local geometric dependencies, whereas it exhibits a Poisson behavior in the regime where the maximum degree diverges more slowly than $\log n$.
\end{abstract}

\section{Introduction}
\subsection{Problem setting and overview}
Let $(X_i)_{i \in \BN}$ be a sequence of independent and identically distributed (i.i.d.) $\BR^d$-valued random variables whose law has a bounded density $f$.
We write its finite essential supremum as $f_{\max}$.
Then, define a point process $\CX_n := \sum_{i=1}^n \delta_{X_i}$ called a binomial process.
Note that $\CX_n$ is simple since the law of $X_1$ admits a density.
For $r > 0$, we define the finite simple graph $G(\CX_n; r)$ by taking the support of the point process $\CX_n$, which is a discrete subset of $\BR^d$, as its vertex set.
For some fixed norm $\|\cdot\|$, two distinct vertices $X_i$ and $X_j$ are joined by an edge whenever $\|X_i-X_j\| \leq r$.
This is called a random geometric graph.
We refer the reader to Penrose~\cite{Penrose.2003} for a comprehensive account of random geometric graphs.
The asymptotic behavior of the sequence of random geometric graphs $G(\CX_n; r_n)$ is frequently studied for a given sequence of positive numbers $(r_n)_{n \in \BN}$.
We will also use geometric graphs based on Poisson point processes.
For a Poisson random variable $N_{\lambda}$ with mean $\lambda$ which is independent of $(X_i)_{i \in \BN}$, define $\CP_{\lambda} := \sum_{i=1}^{N_{\lambda}} \delta_{X_i}$.
$\CP_{\lambda}$ is a Poisson point process with intensity function $\lambda f$.
For $r > 0$, the random graph $G(\CP_{\lambda}; r)$ is defined in the same manner as in the binomial process case.

For the sake of simplicity and intuition, we write $X \in \CX_n$ instead of $X \in \mathrm{supp}(\CX_n)$ (see Section~\ref{section.01}).
For the random geometric graph $G(\CX_n; r_n)$, the maximum degree $\Delta_n$ is defined by
\begin{equation*}
    \Delta_n := \max \{\mathrm{deg}_n(X) : X \in \CX_n\},
\end{equation*}
where $\mathop{\mathrm{deg}_n} (X)$ denotes the degree of the vertex $X$ in $G(\CX_n; r_n)$.
Given a sequence $(k_n)_{n \in \BN}$ of positive integers, let $S_{k_n}(\CX_n)$ be the threshold radius for the maximum degree to be at least $k_n$, defined by
\begin{equation*}
    S_{k_n}(\CX_n) := \inf \{r > 0 : \text{(the maximum degree of $G(\CX_n;r)$)} \geq k_n\}.
\end{equation*}
This is also called the smallest $k_n$-nearest neighbor link in \cite{Penrose.2003}.
Regarding the almost sure limit for $\Delta_n$, it has been shown that if $nr_n^d$ decays faster than $n^{-\varepsilon}$ for some $\varepsilon > 0$, $\Delta_n$ is bounded and consequently no non-trivial limit exists.
For all other regimes satisfying $r_n \to 0$, non-trivial almost sure limits (or limit in probability) are established.
In particular, for the case where $nr_n^d = \Omega(\log n)$, the almost sure limit for $\Delta_n$ is derived from that of $S_{k_n}(\mathcal{X}_n)$, provided that the density $f$ is assumed to have compact support.
This topic has been investigated by various authors, and we refer the reader to Penrose~\cite{Penrose.2003} for a summary.
As for the fluctuations of $\Delta_n$, it has been shown that if $nr_n^d$ decays faster than $n^{-\varepsilon}$ for some $\varepsilon > 0$, $\Delta_n$ concentrates on at most two consecutive bounded integers \cite[Theorem~6.3]{Penrose.2003}.
In the regime where $nr_n^d = n^{o(1)}$ and $nr_n^d = o(\log n)$, it has been established by M\"{u}ller~\cite{Muller.2008} that $\Delta_n$ concentrates on at most two consecutive integers that diverge to infinity.
As for the regime $nr_n^d = \Omega(\log n)$, rigorous fluctuation results have not yet been established; however, fluctuation results are presented in \cite{Muller.2008} without detailed proofs.
For $S_{k_n}(\CX_n)$, the almost sure limit has been established only in the regime where $k_n = \Omega(\log n)$ with a density $f$ having compact support~\cite{Penrose.2003}.
In comparison, the fluctuations of $S_{k_n}(\CX_n)$ is less thoroughly explored across all regimes.

In this paper, we establish limit theorems for the fluctuations of $S_{k_n}(\CX_n)$ in the regime $k_n = o(\log n)$.
Regarding the threshold radius $M_{k_n}(\CX_n)$ for the \emph{minimum} degree, which is defined in a similar manner to $S_{k_n}(\CX_n)$, its fluctuations have been studied for $k_n \equiv k$ (a constant relative to $n$).
The results for $X_1$ following a uniform distribution on $[0,1]^d$ or a standard normal distribution can be found in \cite[Chapter~8]{Penrose.2003}.
In the present work, the fluctuation of $S_{k_n}(\CX_n)$ is characterized, and furthermore, we examine the limiting behavior of the point process formed by the vertices achieving the maximum degree.
In brief, we show that the limiting behavior is Poisson for $k_n \to \infty,~ k_n = o(\log n)$ and compound Poisson for constant $k_n \equiv k$.
This result arises from the local geometric dependencies inherent in random geometric graphs.

\subsection{Notation}\label{section.01}
Let $\mathbb{N}(\BR^d)$ denote the set of non-negative integer-valued locally finite measures on $\BR^d$.
This space is equipped with the $\sigma$-algebra generated by the vague topology, which is induced by the following family of functions:
\begin{equation*}
    \mu \mapsto \int_{\BR^d} f \,d\mu \quad (f \in C_K^+(\BR^d)),
\end{equation*}
where $C_K^+(\BR^d)$ denotes the space of all non-negative continuous functions with compact support.
Although an element of $\mathbb{N}(\BR^d)$ is formally a counting measure, we are often concerned with the inclusion relation between underlying point configurations.
For this reason, for $\CX, \CY \in \mathbb{N}(\BR^d)$, we write $\CX \subset \CY$ for $\CX \leq \CY$ in the sense of measures with a slight abuse of notation.
This inclusion implies $\mathrm{supp}(\CX) \subset \mathrm{supp}(\CY)$; furthermore, $\CX$ is simple whenever $\CY$ is.
In the same spirit, for a simple $\CX \in \mathbb{N}(\BR^d)$, we identify it with its support and write $X \in \CX$ as a shorthand for $X \in \mathrm{supp}(\CX)$ when picking a point $X$ from $\CX$.

A description of the results requires an analysis of the spatial scaling limits of the point processes.
For $a > 0$ and $y \in \BR^d$, we define the dilation operator $D_a$ and the translation operator $T_y$ on $\mathbb{N}(\BR^d)$ as follows:
for a counting measure $\CX = \sum_{\ell=1}^{\CX (\BR^d)} \delta_{x_{\ell}} \in \mathbb{N}(\BR^d)$, we set
\begin{equation*}
    D_a \CX := \sum_{\ell=1}^{\CX (\BR^d)} \delta_{a x_{\ell}}, \quad T_y \CX := \sum_{\ell=1}^{\CX (\BR^d)} \delta_{x_{\ell} + y}.
\end{equation*}
For a simple $\CX$, these can be interpreted as the action of affine transformations on $\CX$, which is naturally identified with a discrete subset of $\BR^d$.
By definition, for all $B \in \CB(\BR^d)$, it holds that
\begin{equation*}
    D_a (T_y \CX)(B) = \CX (a^{-1}B - y).
\end{equation*}

We write $c, c_0, c_1$ for positive constants independent of $n$, whose values may vary line by line.
Throughout this paper, $\theta$ denotes the volume of a $d$-dimensional unit ball, where its dependence on $d$ and the choice of norm is suppressed as it does not affect the analysis.

\section{Main results}
\subsection{The case of degree parameter $k_n$ fixed}
Let $\Gamma$ be a simple graph with $2 \leq j < \infty$ vertices.
Let $h_{\Gamma}$ be the function taking an element $\CY \in \mathbb{N}(\BR^d)$ as its argument, defined by $h_\Gamma(\CY) := \mathbf{1}\{G(\CY; 1) \cong \Gamma\}$, where $\cong$ denotes graph isomorphism.
For a Borel set $B \subset \BR^d$, define
\begin{equation*}
    \mu_{\Gamma,B} := \frac{1}{j!} \int_B f(x)^j \,dx \int_{(\BR^d)^{j-1}} h_{\Gamma}\left(\delta_{\bm{0}} + \sum_{\ell=1}^{j-1} \delta_{x_j}\right) \,\prod_{\ell=1}^{j-1} dx_{\ell}.
\end{equation*}
If $B = \BR^d$, we simply write $\mu_{\Gamma, \BR^d} =: \mu_{\Gamma}$.
$\Gamma$ is said to be feasible if $\mu_\Gamma \neq 0$.
\begin{thm}\label{MainResult.01}
    Assume that $d \geq 1$, $k \geq 1$, and $X_1$ has the bounded density $f$.
    Let $\{\Gamma_i\}_{i=1}^m$ be the set of all non-isomorphic feasible graphs with $k+1$ vertices that have at least one vertex of degree $k$.
    Then, as $n \to \infty$,
    \begin{equation*}
        -n^{(k+1)/(dk)} S_k(\CX_n) \xrightarrow{d} Z_{d,k},
    \end{equation*}
    where $Z_{d,k}$ is a random variable with the Weibull distribution
    \begin{equation*}
        P[Z_{d,k} \leq x] =
        \begin{cases}
            \exp(-\mu_{d,k}(-x)^{dk}) & (x < 0),\\
            1 & (x \geq 0)
        \end{cases}
    \end{equation*}
    and $\mu_{d,k} := \sum_{i=1}^m \mu_{\Gamma_i}$.
\end{thm}
\begin{rem}
    This result implies that there is no non-trivial almost sure limit for $S_k(\CX_n)$.
\end{rem}
As we will see later in the proof of Theorem~\ref{MainResult.01}, the number of vertices of degree $k$ in $G(\CX_n;r_n)$, denoted by $W_{k,n}$, converges not to a Poisson but to a compound Poisson random variable in the regime where $E[W_{k,n}]$ converges to some finite constant and $nr_n^d \to 0$.
This condition is equivalent to the convergence of $n^{k+1} r_n^{dk}$ to some finite constant (see Section~\ref{section.04}).
Moreover, in this regime, $E[\sum_{j \geq k+1} W_{j,n}] \to 0$ holds, implying that with high probability there is no vertex whose degree is strictly greater than $k$ (see Section~\ref{section.04}).
Therefore, to study the behavior of the vertices achieving the maximum degree in this regime, the first step is to analyze the behavior of vertices with degree exactly $k$.
For $k \geq 1$ and $(r_n)_{n \in \BN}$, let $\Phi_{k,n}$ be the point process whose support is the set of vertices of $G(\CX_n; r_n)$ having degree exactly $k$, that is,
\begin{equation*}
    \Phi_{k,n} := \sum_{X \in \CX_n} \delta_X \mathbf{1}\{\mathop{\mathrm{deg}_n} (X) = k\}.
\end{equation*}
\begin{thm}\label{MainResult.05}
    Assume that $d \geq 1$, $k \geq 1$, and $X_1$ has the bounded density $f$.
    Assume also that $(r_n)_{n \in \BN}$ satisfies $n^{k+1} r_n^{dk} \to \beta \in (0,\infty)$.
    Let $\{\Gamma_i\}_{i=1}^m$ be the set of all non-isomorphic feasible graphs with $k+1$ vertices that have at least one vertex of degree $k$.
    Then, $\Phi_{k,n}$ converges in distribution, with respect to the vague topology, to a compound Poisson point process $\widetilde{\CQ}$ with its Laplace functional
    \begin{equation*}
        \CL_{\widetilde{\CQ}}(g) = \exp\left(-\int_{\BR^d} \left(1-\sum_{i=1}^m \frac{\mu_{\Gamma_i}}{\mu_{d,k}} e^{-q_i g(x)}\right) \Lambda(dx)\right),
    \end{equation*}
    where $g : \BR^d \to [0,\infty)$ is measurable and $\Lambda(B) := \beta \sum_{i=1}^m \mu_{\Gamma_i,B}$.
    That is, $\widetilde{\CQ}$ has a Poisson point process with intensity measure $\Lambda$ as its supporting measure, and i.i.d. weights $\zeta_j$ having the distribution $P[\zeta_j = \ell] = (\sum_{1 \leq i \leq m, q_i = \ell} \mu_{\Gamma_i})/\mu_{d,k}$.
\end{thm}
Since the limit of $\Phi_{k,n}$ is Poisson, the number of vertices with degree exactly $k$ is zero with positive probability.
On the event that no vertex has degree $k$, it is necessary to examine the behavior of $\Phi_{k-1,n}$ to understand the vertices achieving the maximum degree.
We present a more general result which covers the regime discussed in Theorem~\ref{MainResult.05}.
\begin{thm}\label{MainResult.06}
    Assume that $d \geq 1$, $k \geq 1$, and $X_1$ has the bounded density $f$.
    Assume also that $n^{k+1} r_n^{dk} \to \infty$ and $nr_n^d \to 0$.
    Let $\{\Gamma_i\}_{i=1}^m$ be the set of all non-isomorphic feasible graphs with $k+1$ vertices that have at least one vertex of degree $k$.
    Then, for every continuity point $x_0$ of $f$, the spatially scaled point process $D_{n^{(k+1)/d} r_n^k} (T_{-x_0} \Phi_{k,n})$ converges in distribution, with respect to the vague topology, to a compound Poisson point process $\widetilde{\CH}_{x_0}$ with its Laplace functional
    \begin{equation*}
        \CL_{\widetilde{\CH}_{x_0}}(g) = \exp\left(-\int_{\BR^d} \left(1-\sum_{i=1}^m \frac{\mu_{\Gamma_i}}{\mu_{d,k}} e^{-q_i g(x)}\right) \lambda_{x_0}\,dx\right),
    \end{equation*}
    where $g : \BR^d \to [0,\infty)$ is measurable and
    \begin{equation*}
        \lambda_{x_0} := \frac{f(x_0)^{k+1}}{(k+1)!} \sum_{i=1}^m \int_{(\BR^d)^k} h_{\Gamma_i} \left(\delta_{\bm{0}} + \sum_{\ell=1}^k \delta_{x_{\ell}}\right) \,\prod_{\ell=1}^k dx_{\ell}.
    \end{equation*}
    That is, $\widetilde{\CH}_{x_0}$ has a homogeneous Poisson point process of intensity $\lambda_{x_0}$ as its supporting measure, and i.i.d. weights $\zeta_j$ having the distribution $P[\zeta_j = \ell] = (\sum_{1 \leq i \leq m, q_i = \ell} \mu_{\Gamma_i})/\mu_{d,k}$.
\end{thm}
\begin{rem}
    If $f(x_0) = 0$ in Theorem~\ref{MainResult.06}, the scaled point process converges to zero.
    This suggests that the scaling centered at $x_0$ is inappropriate.
    However, since our primary objective is to analyze the point configuration that achieves the maximum degree and such an $x_0$ is irrelevant, we do not address this further in the present paper.
\end{rem}
In this regime, the maximum degree $\Delta_n$ is concentrated on two points $\{k-1,k\}$ with high probability~\cite[Theorem~6.3]{Penrose.2003}.
Therefore, Theorems \ref{MainResult.05} and \ref{MainResult.06} allow us to identify the point configuration that achieves the maximum degree.
In particular, if $n^{k+1}r_n^{dk} \to \infty$ and $n^{k+2}r_n^{d(k+1)} \to 0$ for some $k \geq 1$ and $f$ is continuous, then $P[\Delta_n = k] \to 1$ and it follows that
\begin{equation*}
    D_{n^{(k+1)/d} r_n^k} (T_{-x_0} \CM_n) \xrightarrow{d} \widetilde{\CH}_{x_0}
\end{equation*}
for all $x_0 \in \BR^d$, where
\begin{equation*}
    \CM_n := \sum_{X \in \CX_n} \delta_X \mathbf{1}\{\mathop{\mathrm{deg}_n} (X) = \Delta_n\}
\end{equation*}
is the point process whose support is the set of points of $G(\CX_n; r_n)$ that achieves the maximum degree.

\subsection{The case of degree parameter $k_n$ growing with $n$}
\begin{thm}\label{MainResult.02}
    Assume that $d \geq 2$, $(k_n)_{n \in \BN}$ is a $\BN$-valued sequence satisfying $k_n \to \infty$ as $n \to \infty$ and $k_n = o(\log n)$.
    Assume also that $X_1$ has a uniform distribution on the $d$-dimensional unit cube (that is, $f = \mathbf{1}_{[-1/2,1/2]^d}$).
    Let $\theta$ be the volume of a $d$-dimensional unit ball.
    Then, as $n \to \infty$,
    \begin{equation*}
        -en^{1+1/k_n}\theta S_{k_n}(\CX_n)^d \exp\left(-\frac{n\theta S_{k_n}(\CX_n)^d}{k_n}\right) + k_n + \frac{1}{2} \log k_n + \log \sqrt{2\pi} \xrightarrow{d} Z,
    \end{equation*}
    where $Z$ is a random variable with the Gumbel distribution $P[Z \leq x] = \exp(-\exp(-x))$.
\end{thm}
\begin{rem}\label{Remark.01}
    It is little annoying that the term $\exp(-n\theta S_{k_n}(\CX_n)^d/k_n)$ presents in the above statement.
    By restricting our attention to $k_n$ satisfying $k_n n^{-1/k_n} \to 0$, or equivalently $\log k_n - k_n^{-1}\log n \to -\infty$, this can be eliminated and the result is
    \begin{equation*}
        -en^{1+1/k_n}\theta S_{k_n}(\CX_n)^d + k_n + \frac{1}{2} \log k_n + \log \sqrt{2\pi} \xrightarrow{d} Z,
    \end{equation*}
    where $Z$ is a Gumbel random variable.
    For the treatment of this case, see Remark 4.2.
    The assumption on $k_n$ is satisfied, for example, when $k_n \leq \log n/\log \log n$ for all large enough $n$, but it fails when $k_n \geq (1+\varepsilon)\log n/\log \log n ~ (\varepsilon > 0)$ for all large enough $n$.
\end{rem}
The above result can be extended to more general densities.
Consider a class of densities $f$ satisfying following two conditions:
\begin{itemize}[labelsep=1pt, itemsep=-1pt]
    \item For some $\rho_0 > 0$ and $m \in \BN$, $f$ is $C^{2m}$ on $B(\bm{0}; \rho_0)$.
    Further, $\min\{\ell : \partial^{\ell}f (\bm{0}) \neq 0\} = 2m$, $f(\bm{0}) = f_{\max}$, the $2m$-th derivative $\partial^{2m}f(\bm{0})$ is a negative definite multilinear form on $\BR$ and $\bm{0}$ is the unique maximum point of $f$.
    \item For some $\rho_0 > \rho_1 > 0$, there exists $\varepsilon_0 > 0$ such that for all $x \in \BR^d \setminus B(\bm{0}; \rho_1)$, $f(x) \leq (1-\varepsilon_0)f_{\max}$.
\end{itemize}
\begin{thm}\label{MainResult.03}
    Assume that $d \geq 2$, $(k_n)_{n \in \BN}$ is a $\BN$-valued sequence satisfying $k_n \to \infty$ as $n \to \infty$ and $k_n = o(\log n)$.
    Assume also that a density $f$ of $X_1$ satisfies the above two conditions.
    Then, as $n \to \infty$,
    \begin{equation*}
        \begin{aligned}
            -e f_{\max}n^{1+1/k_n}\theta S_{k_n}(\CX_n)^d \exp\left(-\frac{f_{\max}n\theta S_{k_n}(\CX_n)^d}{k_n}\right)&\\
            + k_n + \left(\frac{1+(d/m)}{2}\right)\log k_n + \log \frac{\sqrt{2\pi}}{\gamma f_{\max}^{1+d/(2m)}} &\xrightarrow{d} Z,
        \end{aligned}
    \end{equation*}
    where $\gamma := \int_{\BR^d} \exp(\partial^{2m}f(\bm{0})[z]/((2m)!)) \,dz$ and $Z$ is a random variable with the Gumbel distribution $P[Z \leq x] = \exp(-\exp(-x))$.
\end{thm}
\begin{rem}
    It can be shown by a similar argument to the proof of Theorem~\ref{MainResult.03} that for densities of the form $f(x) = (f_{\max}-\|x\|_2^s) \mathbf{1}_{[-1/2,1/2]^d}(x)$, where $s > 0$ and $\|\cdot\|_2$ denotes the Euclidean norm on $\BR^d$, it holds that
    \begin{equation*}
        -e f_{\max}n^{1+1/k_n}\theta S_{k_n}(\CX_n)^d \exp\left(-\frac{f_{\max}n\theta S_{k_n}(\CX_n)^d}{k_n}\right) + k_n + \left(\frac{1}{2}+\frac{d}{s}\right)\log k_n + \log \frac{\sqrt{2\pi}}{\gamma f_{\max}^{1+d/s}} \xrightarrow{d} Z,
    \end{equation*}
    where $\gamma := \int_{\BR^d} \exp(-\|z\|_2^s) \,dz$ and $Z$ is a random variable with the Gumbel distribution $P[Z \leq x] = \exp(-\exp(-x))$.
\end{rem}
\begin{cor}[Limit in probability for $S_{k_n}(\CX_n)$]\label{Corollary.01}
    Under the assumptions of Theorem \ref{MainResult.02}, \ref{MainResult.03}, 
    \begin{equation*}
        \frac{n^{1+1/k_n}\theta S_{k_n}(\CX_n)^d}{k_n} \to \frac{1}{f_{\max}e}, \quad \text{in probability}
    \end{equation*}
    holds as $n \to \infty$.
\end{cor}

As in the case where $k_n$ is fixed, define
\begin{equation*}
    \Phi_{k_n,n} := \sum_{X \in \CX_n} \delta_X \mathbf{1}\{\mathop{\mathrm{deg}_n} (X) = k_n\}.
\end{equation*}
\begin{thm}\label{MainResult.07}
    Under the assumptions of Theorems \ref{MainResult.02} and \ref{MainResult.03}, assume also that $(r_n)_{n \in \BN}$ is chosen so that $E[\sum_{j=k_n}^\infty W_{j,n}'] \to \beta \in (0,\infty)$ and $n\theta r_n^d = o(k_n)$.
    Then, in the uniform case, $\Phi_{k_n,n}$ converges in distribution, with respect to the vague topology, to a Poisson point process $\CQ$ with intensity measure $\beta \text{Leb}(\cdot \cap [-1/2,1/2]^d)$.
    In the non-uniform case, $D_{k_n^{1/(2m)}}\Phi_{k_n,n}$ converges in distribution, with respect to the vague topology, to a Poisson point process $\CQ$ with intensity measure $\Lambda(B) = \beta \gamma^{-1} \int_B \exp(\partial^{2m}f(\bm{0})[z]/((2m)!)) \,dz$.
\end{thm}

In this regime of $(r_n)_{n \in \BN}$, the number of vertices with degree exceeding $k_n$ converges to zero in probability.
Consequently, when there is a vertex with degree $k_n$, Theorem~\ref{MainResult.07} captures the behavior of the vertices achieving the maximum degree; however, since the limit is Poisson, the number of vertices with degree exactly $k_n$ is zero with positive probability.
On the event that no vertex has degree $k_n$, it is necessary to examine the behavior of $\Phi_{k_n-1,n}$ to understand the vertices achieving the maximum degree.
Müller's result~\cite{Muller.2008} on the two-point concentration of $\Delta_n$ implies that it is sufficient to examine only $\Phi_{k_n,n}$ and $\Phi_{k_n-1,n}$.
\begin{thm}\label{MainResult.08}
    Under the assumptions of Theorem \ref{MainResult.02} and \ref{MainResult.03}, assume also that $(r_n)_{n \in \BN}$ is chosen so that $E[\sum_{j=k_n}^\infty W_{j,n}'] \to \beta \in (0,\infty)$ and $n\theta r_n^d = o(k_n)$.
    Then, $D_{(k_n^{1+d/(2m)}/(n\theta r_n^d))^{1/d}} \Phi_{k_n-1,n}$ converges in distribution, with respect to the vague topology, to a homogeneous Poisson point process $\CH$ of intensity $\beta/f_{\max}$ (in the uniform case, we regard $m = \infty$ and $d/(2m) = 0$).
\end{thm}

\subsection{Heuristics behind the difference in the limiting point configurations}
Theorems~\ref{MainResult.05} and \ref{MainResult.07} show that the point process formed by vertices attaining the maximum degree converges to a compound Poisson process when $k_n$ is fixed, whereas it converges to a Poisson process when $k_n \to \infty$ with $k_n = o(\log n)$.
This difference arises from whether points in the vicinity of a point $X$, conditioned on $X$ attaining the maximum degree, are likely to do so as well (Figure~\ref{Table.01}).

The figure on the left (right) illustrates a point $X$ attaining the maximum degree and its nearest neighbor $Y$ in a random geometric graph $G(\CX_n; r_n)$ with $r_n$ in the regime of Theorem~\ref{MainResult.05} (in the regime of Theorem~\ref{MainResult.07}, respectively).
The circles of radius $r_n$ are centered at $X$ and $Y$.
In both panels, the circles centered at $Y$ are drawn with thicker lines.
Red points represent neighbors unique to $X$, and orange points represent common neighbors of $X$ and $Y$.
\begin{table}[ht]
    \boldmath
    \centering
    \begin{tabular}{cc}
        \includegraphics[width=0.35\linewidth]{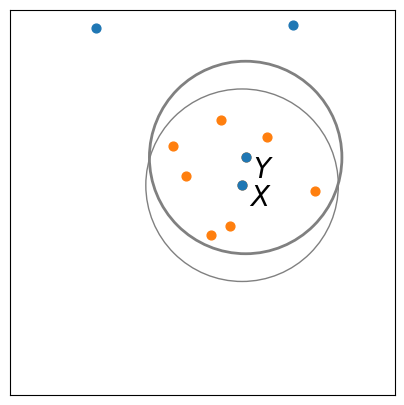} & \includegraphics[width=0.35\linewidth]{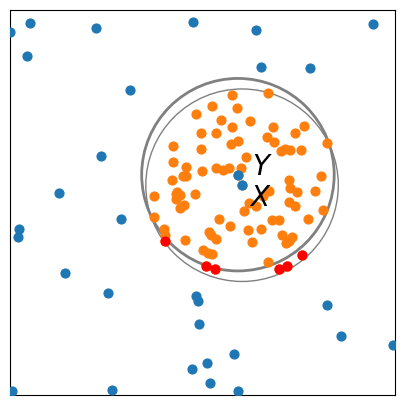}\\
        \large $k_n = k~ (\mathrm{fixed}),$ & \large $k_n \to \infty,~ k_n = o(\log n),$\\
        \large $\mathrm{deg}(X) = k,$ & \large $\mathrm{deg}(X) = k_n,$\\
        \large {\color{red} $\mathrm{deg}(Y) = k$} & \large {\color{red} $\mathrm{deg}(Y) < k_n$}
    \end{tabular}
    \unboldmath
    \captionof{figure}{Illustration of the local structure around a point attaing the maximum degree in a random geometric graph.}\label{Table.01}
\end{table}

In both regimes, the neighborhood of a point achieving the maximum degree is more densely populated than its surroundings.
Since this is a rare event, and remote points in the binomial process behave independently, it is natural to expect that the configuration of such dense clusters is asymptotically a Poisson process.
The limiting point configuration is then determined by how many points with the maximum degree belong to each cluster: if each cluster contains only one such point, the configuration of maximum-degree points follows a Poisson process, whereas it follows a compound Poisson process if multiple such points coexist in each cluster.

For fixed $k_n$, the $r_n$-neighborhood of $X$ contains at most a bounded number of points.
Thus, a point $Y$ near $X$ has the same degree with positive probability; that is, $Y$ also attains the maximum degree with positive probability.
In contrast, when $k_n$ diverges, the $r_n$-neighborhood of $X$ contains a diverging number of points, and $Y$ misses a certain fraction of them.
Since the region outside $X$’s neighborhood is relatively sparse, $Y$ cannot compensate for this substantial loss of points and thus fails to achieve the maximum degree.
This is the geometric mechanism underlying the difference in the limiting behavior in Theorems~\ref{MainResult.05} and \ref{MainResult.07}.

\section{Proof of Theorems \ref{MainResult.01}--\ref{MainResult.06}}\label{section.04}
\subsection{Threshold radius}
Given a sequence of radii $(r_n)_{n \in \BN}$ and an integer $j \geq 0$, let $W_{j,n}$ and $W_{j,n}'$ denote the number of vertices of degree $j$ in $G(\CX_n; r_n)$ and $G(\CP_n; r_n)$, respectively. 
Note that
\begin{equation}
    S_k(\CX_n) > r_n \iff \sum_{j = k}^{n-1} W_{j,n} = 0. \label{equivalence.01}
\end{equation}
Because of the above equivalence, the problem of estimating the fluctuation of $S_k(\CX_n)$ reduces to analyzing the fluctuation of $\sum_{j=k}^{n-1} W_{j,n}$.
We consider a $(r_n)_{n \in \BN}$ satisfying that $E[\sum_{j=k}^{n-1} W_{j,n}]$ converges to some finite constant and this is achieved by taking $n^{k+1}r_n^{dk} = \beta \in (0,\infty)$.
Indeed,
\begin{align*}
    E[W_{k,n}] &= nP[\CX_{n-1}(B(X_n; r_n)) = k]\\
    &= n \int_{\BR^d} \left(\binom{n-1}{k} F(B(x; r_n))^k (1-F(B(x; r_n)))^{n-1-k}\right)f(x)\,dx\\
    &\sim n^{k+1}r_n^{dk} \frac{\theta^k}{k!} \int_{\BR^d} f(x)^{k+1}\,dx\\
    &= \beta \frac{\theta^k}{k!} \int_{\BR^d} f(x)^{k+1}\,dx\\
    &\in (0,\infty)
\end{align*}
and $E[\sum_{j=k+1}^{n-1} W_{j,n}] = O(n^{-1/k})$.

In the proof, we use the subgraph count.
For a feasible graph $\Gamma$, let $G_n(\Gamma)$ denote the induced $\Gamma$-subgraph count, that is, the number of counting measures $\CY \in \mathbb{N}(\BR^d)$ for which $\CY \subset \CX_n$ and $G(\CY; r_n)$ is isomorphic to $\Gamma$.
The next lemma is the Poisson limit theorem for $G_n(\Gamma)$:
\begin{lem}[{Penrose~\cite[Corollary~3.6]{Penrose.2003}}]\label{Lemma.01}
    Let $j \geq 2$ and $\{\Gamma_i\}_{i=1}^m$ be a set of non-isomorphic feasible connected graphs with $j$ vertices.
    Suppose that $(r_n)_{n \in \BN}$ satisfies $n^j r_n^{d(j-1)} \to \beta$ for some $\beta \in (0,\infty)$.
    Let $\{Z_i\}_{i=1}^m$ be a set of independent Poisson random veriables with $E[Z_i] = \beta \mu_{\Gamma_i}$.
    Then,
    \begin{equation*}
        (G_n(\Gamma_1), \cdots, G_n(\Gamma_m)) \xrightarrow{d} (Z_1, \cdots, Z_m)
    \end{equation*}
    holds as $n \to \infty$.
\end{lem}
\begin{proof}[Proof of Theorem 2.1]
    Let $x < 0$ and $\beta = (-x)^{dk}$, and take $(r_n)_{n \in \mathbb{N}}$ such that $n^{k+1}r_n^{dk} = \beta$.
    Let $\{\Gamma_i\}_{i=1}^m$ be the set of all non-isomorphic feasible graphs with $k+1$ vertices that have at least one vertex of degree $k$, and $q_i > 0$ denote the number of vertices of degree $k$ in $\Gamma_i$.
    Then, by Lemma~\ref{Lemma.01} we have
    \begin{equation*}
        W_{k,n} = \sum_{i=1}^m q_i G_n(\Gamma_i) \xrightarrow{d} \sum_{i=1}^m q_i Z_i,
    \end{equation*}
    where $\{Z_i\}_{i=1}^m$ be independent Poisson random variables with $E[Z_i] = \beta \mu_{\Gamma_i}$.
    That is, $W_{k,n}$ converges in distribution to the compound Poisson random variables.
    It also holds that $E[\sum_{j=k+1}^{n-1} W_{j,n}]$ converges to $0$ and thus $\sum_{j=k+1}^{n-1} W_{j,n} \to 0$ in probability.
    Therefore, $\sum_{j=k}^{n-1} W_{j,n}$ converges in distribution to $\sum_{i=1}^m q_i Z_i$, from which we obtain $P[\sum_{j=k}^{n-1} W_{j,n} = 0] \to \exp(-\beta \sum_{i=1}^m \mu_{\Gamma_i})$.
    Hence, by the equivalence (\ref{equivalence.01}),
    \begin{equation*}
        \begin{aligned}
            &P[-n^{(k+1)/(dk)}S_k(\CX_n) < x]\\
            &= P[S_k(\CX_n) > r_n]\\
            &\xrightarrow{n \to \infty} \exp(-\mu_{d,k}(-x)^{dk}).
        \end{aligned}
    \end{equation*}
    For $x \geq 0$, $P[-n^{(k+1)/(dk)}S_k(\CX_n) < x] = 1$ since $S_k(\CX_n) > 0$ with probability $1$. 
\end{proof}

\subsection{Point configuration}
For each $\Gamma_i$ as in Theorem~\ref{MainResult.05}, define
\begin{equation*}
    \begin{aligned}
        \widetilde{\Phi}_{n,\Gamma_i} &:= \sum_{X \in \CX_n} \delta_X \mathbf{1}\{\text{$\exists \CY \in \mathbb{N}(\BR^d)$ such that $\CY \subset \CX_n$, $G(\CY; r_n) \cong \Gamma_i$ and $\min D_{k,n}(\CY) = X$}\},
    \end{aligned}
\end{equation*}
where $D_{k,n}(\CY)$ denotes the set of vertices of degree $k$ in $G(\CY; r_n)$ and the minimum is taken with respect to the lexicographic ordering on $\BR^d$.
Then, $q_i \widetilde{\Phi}_{n,\Gamma_i}(\BR^d)$ equals to the number of vertices of degree $k$ in $\Gamma_i$-subgraphs.
When $n^{k+1}r_n^{dk} \to \beta \in (0,\infty)$, with high probability, there is no connected subgraph of order $k+2$ or greater.
Thus, with high probability, $\Phi_{k,n}(\BR^d) = \sum_{i=1}^m q_i \widetilde{\Phi}_{n,\Gamma_i}(\BR^d)$ holds.
Therefore, it suffices to analyze $\widetilde{\Phi}_{n,\Gamma_i}$.
\begin{lem}[{Penrose~\cite[Theorem~2.3]{Penrose.2003}}]\label{Lemma.05}
    Let $(I, \sim)$ be a finite simple graph and, $(\xi_i)_{i \in I}$ be a collection of Bernoulli variables such that $\xi_i$ is independent from $\xi_j$ whenever $i \nsim j$.
    That is, assume $(I, \sim)$ is a dependency graph for $(\xi_i)_{i \in I}$.
    Let $\{I(j)\}_{j=1}^{\ell}$ be a partition of $I$ and for $1 \leq j \leq \ell$, set $W_j := \sum_{i \in I(j)} \xi_i$.
    Let $(Z_i)_{i=1}^{\ell}$ be a set of independent Poisson variables with $E[Z_j] = E[W_j]$.
    Then, for $p_i := E[\xi_i]$ and $p_{ij} := E[\xi_i \xi_j]$,
    \begin{equation*}
        d_{\text{TV}}\left((W_j)_{1 \leq j \leq \ell}, (Z_j)_{1 \leq j \leq \ell}\right) \leq 3\left(\sum_{i \in I} \sum_{j \in \CN_i \setminus \{i\}}p_{ij} + \sum_{i \in I} \sum_{j \in \CN_i} p_i p_j\right),
    \end{equation*}
    where $\CN_i := \{i\} \cup \{j \in I : j \sim i\}$ is the neighborhood of $i$ with respect to the adjacency relation $\sim$.
\end{lem}
Let $\mathcal{C}$ be the set of all bounded measurable sets in $\BR^d$.
For a point process $\Psi$ on $\BR^d$, we define $\mathcal{C}_{\Psi} := \{C \in \mathcal{C} : \Psi(\partial C) = 0 ~\text{a.s.}\}$.
By Kallenberg's result~\cite[Theorem~16.16]{Kallenberg.2002}, the following equivalence holds:
\begin{equation*}
    \begin{aligned}
        \Psi_n \xrightarrow{d} \Psi &\iff \text{$(\Psi_n(C_1), \cdots, \Psi_n(C_s)) \xrightarrow{d} (\Psi(C_1), \cdots, \Psi(C_s))$ for all $\{C_j\}_{j=1}^s \subset \mathcal{C}_{\Psi}$},~ s \in \BN_{>0}\\
        &\iff \text{$\int_{\BR^d} g \,d\Psi_n \xrightarrow{d} \int_{\BR^d} g \,d\Psi$ for all $g \in C_K^+(\BR^d)$},
    \end{aligned}
\end{equation*}
where $C_K^+(\BR^d)$ denotes the space of all non-negative continuous functions with compact support.
Note that if we let $\Omega := \mathop{\mathrm{supp}} f = \mathop{\mathrm{supp}} \Lambda$, then $\mathcal{C}_{\widetilde{\CQ}} = \{C \in \mathcal{C} : \Lambda(\Omega \cap \partial C) = 0\}$.
\begin{prop}\label{Proposition.03}
    Assume that $n^{k+1}r_n^{dk} \to \beta$ for some $\beta \in (0,\infty)$.
    Let $\{\Gamma_i\}_{i=1}^m$ be as in Theorem~\ref{MainResult.05}.
    Assume that $\{C_j\}_{j=1}^s \subset \mathcal{C}_{\widetilde{\CQ}}$ is mutually disjoint.
    Then, the joint distribution of $(\widetilde{\Phi}_{n,\Gamma_i}(C_j))_{i,j}$ converges weakly to those of $(Z_{i,j})_{i,j}$ which is a collection of independent Poisson random variables with $E[Z_{i,j}] = \beta \mu_{\Gamma_i, C_j}$.
\end{prop}
\begin{proof}
    Let $I_n$ be a set of indices $(\bm{\ell},i,j)$ where $\bm{\ell}$ runs over all subsets of $\{1, 2, \cdots, n\}$ of size $k+1$.
    Define the adjacency relation $\sim$ on $I_n$ by $(\bm{\ell},i,j) \sim (\bm{\ell}',i',j') \iff \bm{\ell} \cap \bm{\ell}' \neq \emptyset \land (\bm{\ell},i,j) \neq (\bm{\ell}',i',j')$.
    Set $h_{n,\Gamma_i,C_j}(\CY) := \mathbf{1}\{G(\CY;r_n) \cong \Gamma_i\} \cap \{\min D_{k,n}(\CY) \in C_j\}$ and $\xi_{\bm{\ell},i,j} := h_{n,\Gamma_i,C_j}(\sum_{\ell \in \bm{\ell}} \delta_{X_{\ell}})$.
    Then $(I_n,\sim)$ is a dependency graph for $(\xi_{\bm{\ell},i,j})$ and $\widetilde{\Phi}_{n,\Gamma_i}(C_j) = \sum_{\bm{\ell}} \xi_{\bm{\ell},i,j}$.

    Since each $\Gamma_i$ is connected and of order $k+1$, $\xi_{(\bm{\ell},i,j)}$ vanishes whenever there exists a pair of points whose distance exceeds $kr_n$.
    In particular, in order for $\xi_{(\bm{\ell},i,j)}$ to be nonzero, the distance from a given vertex to all other vertices must be at most $kr_n$.
    It shows that $E[\xi_{(\bm{\ell},i,j)}] \leq (f_{\max} \theta k r_n)^k$.
    Since
    \begin{equation*}
        \text{card}(\CN_{(\bm{\ell},i,j)}) \leq ms\left(\binom{n}{k+1} - \binom{n-k-1}{k+1}\right) \leq cn^k
    \end{equation*}
    holds, we can estimate
    \begin{equation*}
        \sum_{(\bm{\ell},i,j) \in I_n} \sum_{(\bm{\ell}',i',j') \in \CN_{(\bm{\ell},i,j)}} E[\xi_{(\bm{\ell},i,j)}] E[\xi_{(\bm{\ell}',i',j')}] \leq cn^{2k+1} r_n^{2dk} = cn^{k+1} r_n^{dk} (nr_n^d)^k \to 0
    \end{equation*}
    because $nr_n^d \to 0$ and $k \geq 1$.
    Recall that $\{\Gamma_i\}_{i=1}^m$ is non-isomorphic and $\{C_j\}_{j=1}^s$ is mutually disjoint.
    If $\bm{\ell} = \bm{\ell}'$, then $\xi_{(\bm{\ell},i,j)} \xi_{(\bm{\ell}',i',j')} = 0$ unless $(i,j) = (i',j')$.
    Thus, $E[\xi_{(\bm{\ell},i,j)}\xi_{(\bm{\ell}',i',j')}]$ does not affect the upper bound unless $\bm{\ell} \neq \bm{\ell}'$.
    If $\bm{\ell} \cap \bm{\ell}'$ is of size $h \in \{1, 2, \cdots, k\}$, then
    \begin{equation*}
        E[\xi_{(\bm{\ell},i,j)}\xi_{(\bm{\ell}',i',j')}] \leq (f_{\max}\theta k^d r_n^d)^{2k-h+1}.
    \end{equation*}
    The number of such pairs $((\bm{\ell},i,j),(\bm{\ell}',i',j'))$ is $ms\binom{n}{k+1}\binom{k+1}{h}\binom{n-k-1}{k-h+1} \leq cn^{2k-h+2}$, and thus
    \begin{equation*}
        \sum_{(\bm{\ell},i,j) \in I_n} \sum_{(\bm{\ell}',i',j') \in \CN_{(\bm{\ell},i,j)} \setminus \{(\bm{\ell},i,j)\}} E[\xi_{(\bm{\ell},i,j)} \xi_{(\bm{\ell}',i',j')}] \leq c\sum_{h=1}^k n^{k+1}r_n^{dk} (nr_n^d)^{k-h+1}
    \end{equation*}
    since $nr_n^d \to 0$.
    Therefore, by Lemma \ref{Lemma.05}, $d_{\text{TV}}((\widetilde{\Phi}_{n,\Gamma_i}(C_j))_{i,j}, (\text{Po}(E[\widetilde{\Phi}_{n,\Gamma_i}(C_j)]))_{i,j}) \to 0$.

    It remains to show that $E[\widetilde{\Phi}_{n,\Gamma_i}(C_j)] \to \beta \mu_{\Gamma_i,C_j}$.
    For an open set $O \subset \BR^d$ with $\text{Leb}(\partial O) = 0$, $E[\widetilde{\Phi}_{n,\Gamma_i}(O)] \to \beta \mu_{\Gamma_i,O}$ follows immediately from an argument analogous to that of Penrose~\cite[Proposition~3.1]{Penrose.2003}.
    By considering partitions of $\BR^d$ into cubes, we can obtain sequences of open sets $\{O_\ell\}_{\ell \in \BN}$ and $\{O_\ell'\}_{\ell \in \BN}$ such that $\text{Leb}(\partial O_{\ell}) = \text{Leb}(\partial O_{\ell}') = 0$ and $\bigcup_{\ell \in \BN} O_{\ell} = \text{int}(C_j),~ \bigcap_{\ell \in \BN} O_{\ell}' = \text{cl}(C_j)$.
    Since $\widetilde{\Phi}_{n,\Gamma_i}(\cdot)$ has monotonicity, for all $\ell \in \BN$,
    \begin{equation*}
        \beta \mu_{\Gamma_i, O_\ell} \leq \liminf_{n \to \infty} E[\widetilde{\Phi}_{n,\Gamma_i}(C_j)] \leq \limsup_{n \to \infty} E[\widetilde{\Phi}_{n,\Gamma_i}(C_j)] \leq \beta \mu_{\Gamma_i, O_\ell '}
    \end{equation*}
    holds.
    By continuity of measure,
    \begin{equation*}
        \beta \mu_{\Gamma_i, \text{int}(C_j)} \leq \liminf_{n \to \infty} E[\widetilde{\Phi}_{n,\Gamma_i}(C_j)] \leq \limsup_{n \to \infty} E[\widetilde{\Phi}_{n,\Gamma_i}(C_j)] \leq \beta \mu_{\Gamma_i, \text{cl}(C_j)}
    \end{equation*}
    holds.
    From the assumption $C_j \in \mathcal{C}_{\widetilde{\CQ}}$, the left-most and right-most sides are equal to $\beta \mu_{\Gamma_i,C_j}$.
\end{proof}
\begin{proof}[Proof of Theorem~\ref{MainResult.05}]
    Without loss of generality, we can assume $\{C_j\}_{j=1}^s \subset \mathcal{C}_{\widetilde{\CQ}}$ is mutually disjoint.
    Then by Proposition~\ref{Proposition.03}, it follows that
    \begin{equation*}
        \left(\sum_{i=1}^m q_i \widetilde{\Phi}_{n,\Gamma_i}(C_1), \cdots, \sum_{i=1}^m q_i \widetilde{\Phi}_{n,\Gamma_i}(C_s)\right) \xrightarrow{d} \left(\sum_{i=1}^m q_i \CQ_i (C_1), \cdots, \sum_{i=1}^m q_i \CQ_i (C_s)\right),
    \end{equation*}
    where $(\CQ_i)_{i=1}^m$ are independent Poisson point processes and $\CQ_i$ has the intensity measure $\Lambda_i(B) = \beta \mu_{\Gamma_i,B}$.
    Thus, $\sum_{i=1}^m q_i \widetilde{\Phi}_{n,\Gamma_i} \xrightarrow{d} \sum_{i=1}^m q_i \CQ_i$ and for all $g \in C_K^+(\BR^d)$,
    \begin{equation*}
        \sum_{i=1}^m q_i \int_{\BR^d} g \,d\widetilde{\Phi}_{n,\Gamma_i} \xrightarrow{d} \sum_{i=1}^m q_i \int_{\BR^d} g \,d\CQ_i.
    \end{equation*}
    Since $g$ is uniformly continuous and compactly supported, there exist a sequence $(a_n)_{n \in \BN}$ satisfying $a_n \to 0$ and $R > 0$ such that, with high probability,
    \begin{equation*}
        \left|\int_{\BR^d} g \,d\Phi_{k,n} - \sum_{i=1}^m q_i \int_{\BR^d} g \,d\widetilde{\Phi}_{n,\Gamma_i}\right| \leq a_n \sum_{i=1}^m q_i \widetilde{\Phi}_{n,\Gamma_i}(B(\bm{0};R))
    \end{equation*}
    holds.
    From the tightness of $\sum_{i=1}^m q_i \widetilde{\Phi}_{n,\Gamma_i}(B(\bm{0};R))$, we have
    \begin{equation*}
        \int_{\BR^d} g \,d\Phi_{k,n} - \sum_{i=1}^m q_i \int_{\BR^d} g \,d\widetilde{\Phi}_{n,\Gamma_i} \to 0, \quad \text{in probability}.
    \end{equation*}
    Therefore $\int_{\BR^d} g \,d\Phi_{k,n} \xrightarrow{d} \sum_{i=1}^m q_i \int_{\BR^d} g \,d\CQ_i$ and since $g \in C_K^+(\BR^d)$ is arbitrary, $\Phi_{k,n} \xrightarrow{d} \sum_{i=1}^m q_i \CQ_i$.
    Finally, if we define $\widetilde{\CQ}' := \sum_{i=1}^m q_i \CQ_i$, then
    \begin{equation*}
        \begin{aligned}
            \CL_{\widetilde{\CQ}'}(g) &= E\left[\exp\left(-\sum_{i=1}^m q_i \int_{\BR^d} g \,d\CQ_i\right)\right]\\
            &= \prod_{i=1}^m E\left[\exp\left(-\int_{\BR^d} q_i g \,d\CQ_i\right)\right]\\
            &= \prod_{i=1}^m \exp\left(-\int_{\BR^d} (1-e^{-q_i g}) \,d\Lambda_i\right)\\
            &= \exp\left(-\int_{\BR^d}\left(1-\sum_{i=1}^m \frac{\mu_{\Gamma_i}}{\mu_{d,k}} e^{q_i g}\right) \,d\Lambda\right)
        \end{aligned}
    \end{equation*}
    holds since the Radon-Nikodym derivative of $\Lambda_i$ with respect to $\Lambda$ is $\mu_{\Gamma_i}/\mu_{d,k}$ by definition.
    Thus $\widetilde{\CQ}' \overset{d}{=} \widetilde{\CQ}$.
\end{proof}
Let us move on to the proof of Theorem~\ref{MainResult.06}.
Since $\widetilde{\CH}_{x_0}$ has a homogeneous Poisson process as its supporting measure, it holds that $\CC_{\widetilde{\CH}_{x_0}} = \{C \in \CC : \text{Leb}(\partial C) = 0\}$.
\begin{prop}\label{Proposition.04}
    Assume that $n^{k+1}r_n^{dk} \to \infty$.
    Let $\{\Gamma_i\}_{i=1}^m$ be as in Theorem~\ref{MainResult.06}.
    Assume that $\{C_j\}_{j=1}^s \subset \CC_{\widetilde{\CH}_{x_0}}$ is mutually disjoint.
    Then, the joint distribution of
    \begin{equation*}
        \left(\left(D_{n^{(k+1)/d}r_n^k}(T_{-x_0} \widetilde{\Phi}_{n,\Gamma_i})\right)(C_j)\right)_{i,j} = \left(\widetilde{\Phi}_{n,\Gamma_i}\left(n^{-(k+1)/d}r_n^{-k}C_j + x_0\right)\right)_{i,j}
    \end{equation*}
    converges weakly to those of $(Z_{i,j})_{i,j}$ which is a set of independent Poisson random variables with
    \begin{equation*}
        E[Z_{i,j}] = \frac{f(x_0)^{k+1}}{(k+1)!} \int_{(\BR^d)^k} h_{\Gamma_i}\left(\delta_{\bm{0}} + \sum_{\ell=1}^k \delta_{x_{\ell}}\right) \,\prod_{\ell=1}^k dx_{\ell} \times \text{Leb}(C_j).
    \end{equation*}
\end{prop}
\begin{proof}
    The proof of Poisson approximation is similar to that of Proposition~\ref{Proposition.03}.
    Let
    \begin{equation*}
        \xi_{\bm{\ell},i,j} := h_{n,\Gamma_i,(n^{-(k+1)/d}r_n^{-k}C_j + x_0)} \left(\sum_{\ell \in \bm{\ell}} \delta_{X_\ell}\right).
    \end{equation*}
    Then $\widetilde{\Phi}_{n,\Gamma_i}\left(n^{-(k+1)/d}r_n^{-k}C_j + x_0\right) = \sum_{\bm{\ell}} \xi_{\bm{\ell},i,j}$ and since $\xi_{\bm{\ell},i,j} = 0$ unless at least one point of $\{X_\ell : \ell \in \bm{\ell}\}$ lies in $n^{-(k+1)/d}r_n^{-k}C_j + x_0$, it follows that
    \begin{equation*}
        \sum_{(\bm{\ell},i,j) \in I_n} \sum_{(\bm{\ell}',i',j') \in \CN_{(\bm{\ell},i,j)}} E[\xi_{(\bm{\ell},i,j)}] E[\xi_{(\bm{\ell}',i',j')}] \leq n^{-2(k+1)/d}r_n^{-2k} \times cn^{k+1}r_n^{dk} (nr_n^d)^k.
    \end{equation*}
    The right-hand side of the above tends to $0$ by the assumption $n^{k+1}r_n^{dk} \to \infty$.
    Also,
    \begin{equation*}
        \sum_{(\bm{\ell},i,j) \in I_n} \sum_{(\bm{\ell}',i',j') \in \CN_{(\bm{\ell},i,j)} \setminus \{(\bm{\ell},i,j)\}} E[\xi_{(\bm{\ell},i,j)} \xi_{(\bm{\ell}',i',j')}] \leq n^{-(k+1)/d}r_n^{-k} \times c\sum_{h=1}^k n^{k+1}r_n^{dk} (nr_n^d)^{k-h+1}
    \end{equation*}
    and this tends to $0$ since $nr_n^d \to 0$.
    Thus by Lemma~\ref{Lemma.05}, $(\widetilde{\Phi}_{n,\Gamma_i}(n^{-(k+1)/d}r_n^{-k}C_j + x_0))_{i,j}$ is approximated by independent Poisson variables with the same expectations.

    In what follows, we investigate the asymptotic behavior of the expectations
    \begin{align*}
        &E\left[\widetilde{\Phi}_{n,\Gamma_i}\left(n^{-(k+1)/d}r_n^{-k}C_j + x_0\right)\right]\\
        &= \binom{n}{k+1} \int_{(\BR^d)^{k+1}} h_{n,\Gamma_i,(n^{-(k+1)/d}r_n^{-k}C_j + x_0)} \left(\sum_{\ell=1}^{k+1} \delta_{x_{\ell}}\right) \prod_{\ell=1}^{k+1} f(x_\ell)\,dx_\ell.
    \end{align*}
    By the change of variables $x_\ell = x_{k+1} + r_n y_\ell$ for $1\leq \ell \leq k$ and $x_{k+1} = n^{-(k+1)/d}r_n^{-k}x + x_0$, the above equals
    \begin{align*}
        \binom{n}{k+1} n^{-(k+1)}r_n^{-dk} r_n^{dk} \int_{(\BR^d)^{k+1}} h_{n,\Gamma_i,(n^{-(k+1)/d}r_n^{-k}C_j)} \left(\delta_{(n^{-(k+1)/d}r_n^{-k}x)} + \sum_{\ell=1}^k \delta_{(n^{-(k+1)/d}r_n^{-k}x + r_n y_{\ell})}\right) & \\
        \times f(x_0+n^{(k+1)/d}r_n^{-k}x) \,dx \prod_{\ell=1}^k f(x_0+n^{-(k+1)/d}r_n^{-k}x+r_ny_\ell)\,dy_\ell.&
    \end{align*}
    If $x \in \text{int}(C_j)$, the indicator function in the integral is equals to $h_{\Gamma_i}(\delta_{\bm{0}} + \sum_{\ell=1}^k \delta_{y_{\ell}})$ for all large enough $n$ since $r_n \ll n^{-(k+1)/d}r_n^{-k}$ by the assumption.
    On the other hand, if $x \in \text{ext}(C_j)$, the indicator function equals $0$ for all large enough $n$.
    Since $\text{Leb}(\partial C_j) = 0$ and the integrand is $0$ except for $(x,y_1, \cdots, y_k) \in B(\bm{0}; R)^{k+1}$ for some large $R>0$, by the dominated convergence theorem, the above is asymptotic to
    \begin{equation*}
        \frac{f(x_0)^{k+1}}{(k+1)!} \int_{(\BR^d)^k} h_{\Gamma_i} \left(\delta_{\bm{0}} + \sum_{\ell=1}^k \delta_{y_{\ell}}\right) \,\prod_{\ell=1}^k dy_{\ell} \times \text{Leb}(C_j),
    \end{equation*}
    and this completes the proof.
\end{proof}
\begin{proof}[Proof of Theorem~\ref{MainResult.06}]
    Following the argument of the proof of Theorem~\ref{MainResult.05},
    \begin{equation*}
        D_{n^{(k+1)/d}r_n^{k}}(T_{-x_0} \Phi_{k,n}) \xrightarrow{d} \sum_{i=1}^m q_i \CH_{x_0,i},
    \end{equation*}
    where $(\CH_{x_0,i})_{i=1}^m$ are independent homogeneous Poisson point processes and $\CH_{x_0, i}$ is of intensity
    \begin{equation*}
        \frac{f(x_0)^{k+1}}{(k+1)!} \int_{(\BR^d)^k} h_{\Gamma_i}\left(\delta_{\bm{0}} + \sum_{\ell=1}^k \delta_{y_{\ell}}\right) \,\prod_{\ell=1}^k dy_{\ell}.
    \end{equation*}
    By calculating the Laplace functional of $\sum_{i=1}^m q_i \CH_{x_0,i}$, it follows that $\sum_{i=1}^m q_i \CH_{x_0,i} \overset{d}{=} \widetilde{\CH}_{x_0}$, since
    \begin{equation*}
        \frac{\mu_{\Gamma_i}}{\mu_{d,k}} = \frac{\int_{(\BR^d)^k} h_{\Gamma_i}(\delta_{\bm{0}} + \sum_{\ell=1}^k \delta_{x_{\ell}}) \,\prod_{\ell=1}^k dx_{\ell}}{\sum_{j=1}^m \int_{(\BR^d)^k} h_{\Gamma_j}(\delta_{\bm{0}} + \sum_{\ell=1}^k \delta_{x_{\ell}}) \,\prod_{\ell=1}^k dx_{\ell}}
    \end{equation*}
    holds.
\end{proof}

\section{Proofs of Theorems \ref{MainResult.02} and \ref{MainResult.03}}
\subsection{Threshold radius}
Recall that we now restrict attention to the case where $X_1$ is uniformly distributed on the $d$-dimensional unit cube.
As in the proof of Theorem \ref{MainResult.01}, Note that
\begin{equation}
    S_{k_n}(\CX_n) > r_n \iff \sum_{j=k_n}^{n-1} W_{j,n}.
\end{equation}
We consider an $(r_n)_{n \in \BN}$ satisfying that $E[\sum_{j=k_n}^{n-1} W_{j,n}]$ converges to some finite constant.
In this section, we analyze $\sum_{j=k_n}^\infty W_{j,n}'$ rather than $\sum_{j=k_n}^{n-1} W_{j,n}$.
If we take $n\theta r_n^d = \Omega(k_n)$, then $E[\sum_{j=k_n}^\infty W_{j,n}']$ fails to converge.
Therefore, we would take $(r_n)_{n \in \BN}$ such that
\begin{equation}
    n\theta r_n^d = o(k_n), \quad E[W_{k_n,n}'] \sim n\frac{(n\theta r_n^d)^{k_n}}{k_n!} e^{-n\theta r_n^d} \to \beta \label{condition.01}
\end{equation}
for some $\beta \in (0,\infty)$.
By Stirling's formula, the second condition of (\ref{condition.01}) is equivalent to
\begin{equation*}
    \frac{n}{\sqrt{2\pi k_n}} \left(\frac{en\theta r_n^d}{k_n}\right)^{k_n} e^{-n\theta r_n^d} \to \beta.
\end{equation*}
This condition is satisfied by choosing $(r_n)_{n \in \BN}$ as the solution of the following equation:
\begin{equation*}
    \frac{n\theta r_n^d}{k_n} = \frac{1}{e} \left(\frac{n}{\beta \sqrt{2\pi k_n}}\right)^{-1/k_n} \exp\left(\frac{n\theta r_n^d}{k_n}\right).
\end{equation*}
Since we now assume $k_n = o(\log n)$, it follows that $(n/\sqrt{2\pi k_n})^{-1/k_n} = o(1)$.
Thus it can be expressed as
\begin{equation}
    \frac{n\theta r_n^d}{k_n} = -W_0\left(-\frac{1}{e}\left(\frac{n}{\beta \sqrt{2\pi k_n}}\right)^{-1/k_n}\right), \label{condition.02}
\end{equation}
where $W_0$ is the principal branch of the Lambert $W$ function.
Then, $n\theta r_n^d = o(k_n)$ since $\lim_{t \to 0} W_0(t) = 0$.
Note that the assumptions $n\theta r_n^d = o(k_n)$ and $k_n = o(\log n)$ imply $r_n^d \ll \log n/n$.
Let $\text{Po}(\lambda)$ be a Poisson random variable with mean $\lambda$ and then, it holds that
\begin{equation*}
    \begin{aligned}
        E\left[\sum_{j = k_n+1}^\infty W_{j,n}'\right] &\leq nP[\text{Po}(n\theta r_n^d) \geq k_n+1]\\
        &= O(n\theta r_n^d/k_n) \times nP[\text{Po}(n\theta r_n^d) = k_n]\\
        &\to 0.
    \end{aligned}
\end{equation*}
Therefore, $\sum_{j = k_n+1}^\infty W_{j,n}'$ converges to $0$ in probability.

We now establish a Poisson limit theorem for $W_{k_n,n}'$.
For this purpose, we use the following result:
\begin{lem}[{Reformulation of Penrose~\cite[Theorem~6.7]{Penrose.2003}}]\label{Lemma.02}
    Let $A \subset \BZ_{\geq 0}$ and, $(r_n)_{n \in \BN}$ and $(\lambda(n))_{n \in \BN}$ be sequences of positive numbers.
    Define
    \begin{align*}
        &I_1(n) := \lambda(n)^2 \int_{\BR^d} F(dx) \int_{B(x; 3r_n)} F(dy) P[\CP_{\lambda(n)}(B_x^n) \in A]P[\CP_{\lambda(n)}(B_y^n) \in A],\\
        &I_2(n) := \lambda(n)^2 \int_{\BR^d} F(dx) \int_{B(x; 3r_n)} F(dy) P[\{\CP_{\lambda(n)}^y(B_x^n) \in A\} \cap \{\CP_{\lambda(n)}^x(B_y^n) \in A\}],
    \end{align*}
    where $B_x^n := B(x;r_n)$ and $\CP_{\lambda(n)}^x := \CP_{\lambda(n)} + \delta_x$.
    Then,
    \begin{equation*}
        d_\text{TV}(W_{A,\lambda(n)}', \text{Po}(E[W_{A,\lambda(n)}'])) \leq \min\left(3, \frac{1}{E[W_{A,\lambda(n)}']}\right) (I_1(n) + I_2(n))
    \end{equation*}
    holds, where $d_\text{TV}$ denotes the total variation distance and $W_{A,\lambda(n)}' := \sum_{j \in A} W_{j,\lambda(n)}'$.
\end{lem}
\begin{rem}
    In \cite[Theorem~6.7]{Penrose.2003}, the proof is provided under the assumption that the density $f$ of $X_1$ is almost everywhere continuous.
    We can dispense with this assumption by using the Lebesgue density theorem, and thus we state Lemma~\ref{Lemma.02} for a general density $f$.
    The proof for this general case is essentially contained in the proof of Proposition~\ref{Proposition.05} in this paper, to which the interested reader is referred.
\end{rem}
\begin{lem}\label{Lemma.03}
    Let $(\lambda(n))_{n \in \BN}$ satisfy $n-n^{3/4} \leq \lambda(n) \leq n+n^{3/4}$ for all $n \in \BN$.
    Suppose that $X_1$ is uniformly distributed on the $d$-dimensional unit cube.
    Then, for the sequence $(r_n)_{n \in \BN}$ taken in (\ref{condition.02}), $W_{k_n,\lambda(n)}' \xrightarrow{d} \text{Po}(\beta)$ holds.
\end{lem}
\begin{proof}
    By Lemma \ref{Lemma.02}, we have
    \begin{equation*}
        d_\text{TV}(W_{k_n,\lambda(n)}', \text{Po}(E[W_{k_n,\lambda(n)}'])) \leq 3(I_1(n)+I_2(n)).
    \end{equation*}
    Since $E[W_{k_n, \lambda(n)}] \to \beta$, it suffices to show that $I_1(n), I_2(n) \to 0$.

    By the definition of $I_1(n)$ and the condition on $(\lambda(n))_{n \in \BN}$,
    \begin{equation*}
        I_1(n) \leq \left(\lambda(n)\frac{(\lambda(n)\theta r_n^d)^{k_n}}{k_n!} e^{-\lambda(n)\theta r_n^d}\right)^2 3^d r_n^d \leq cr_n^d \to 0.
    \end{equation*}
    Thus, it remains to show that $I_2(n) \to 0$.
    For $x,y \in \BR^d$, define
    \begin{equation*}
        h_n(x,y) := P[\{\CP_{\lambda(n)r_n^d}^y(B(x;1)) = k_n\} \cap \{\CP_{\lambda(n)r_n^d}^x(B(y;1)) = k_n\}],
    \end{equation*}
    where $\CP_{\lambda(n)r_n^d}$ denotes a Poisson point process with intensity $\lambda(n)r_n^d \mathbf{1}_{[-1/(2r_n), 1/(2r_n)]^d}$.
    Then by the change of variables,
    \begin{equation*}
        I_2(n) = \lambda(n)^2 r_n^{2d} \int_{[-1/(2r_n), 1/(2r_n)]^d} \int_{B(x;3)} h_n(x,y) \,dydx.
    \end{equation*}
    In what follows, we write $B_x := B(x;1)$ and $F_n(\cdot) := \text{Leb}(\cdot \cap [-1/(2r_n), 1/(2r_n)]^d)$.

    If $y \in B_x \setminus \{x\}$, then
    \begin{equation*}
        P[\CP_{\lambda(n)r_n^d}^y(B_x) = k_n \mid \CP_{\lambda(n)r_n^d}^x(B_y) = k_n] = P[\CP_{\lambda(n)r_n^d}(B_x) = k_n-1 \mid \CP_{\lambda(n)r_n^d}(B_y) = k_n-1]
    \end{equation*}
    holds.
    It can be readily shown that conditionally on $\CP_{\lambda(n)r_n^d} (B_y) = k_n-1$,
    \begin{itemize}[labelsep=1pt, itemsep=-1pt]
        \item the number of points in $B_x \cap B_y$ obeys the binomial distribution $\mathrm{Bi}(k_n-1, F_n(B_x \cap B_y)/F_n(B_y))$,
        \item the number of points in $B_x \setminus B_y$ obeys $\mathrm{Po}(F_n(B_x \setminus B_y))$,
    \end{itemize}
    and these are independent of each other.
    Therefore, we can rewrite the above conditional probability as
    \begin{equation*}
        \begin{aligned}
            &P[\mathrm{Bi}(k_n-1, F_n(B_y \cap B_x)/F_n(B_y)) + \mathrm{Po}(\lambda(n) r_n^d F_n(B_x \setminus B_y)) = k_n-1]\\
            &= P[k_n-1 - \mathrm{Bi}(k_n-1, F_n(B_y \setminus B_x)/F_n(B_y)) + \mathrm{Po}(\lambda(n)r_n^dF_n(B_x \setminus B_y)) = k_n-1],
        \end{aligned}
    \end{equation*}
    Set $F_{y\setminus x}^n := F_n(B_y \setminus B_x)/F_n(B_y)$.
    If $\text{Bi}(k_n-1, F_{y \setminus x}^n) > k_n F_{y \setminus x}^n/2$ and $\text{Po}(\lambda(n)\theta r_n^d F_{x \setminus y}^n) < k_n F_{y \setminus x}^n/2$, then the event in the above fails to occur.
    Therefore, its probability is bounded by
    \begin{equation*}
        P[\text{Bi}(k_n-1, F_{y \setminus x}^n) \leq k_n F_{y \setminus x}^n/2] + P[\text{Po}(\lambda(n)\theta r_n^d F_{x \setminus y}^n) \geq k_n F_{y \setminus x}^n/2].
    \end{equation*}
    By the tail estimates for binomial and Poisson distributions (see Appendix), this is at most
    \begin{align*}
        \exp\left(-k_n F_{x \setminus y}^n H\left(\frac{1}{2}\right)\right) + \exp\left(-\lambda(n)\theta r_n^d F_{x \setminus y}^n H\left(\frac{k_n F_{y \setminus x}^n}{2\lambda(n)\theta r_n^d F_{x \setminus y}^n}\right)\right)
    \end{align*}
    for all $n$ sufficiently large.
    Finally, the condition $E[W_{k_n,\lambda(n)}'] \to \beta$ (recall (\ref{condition.01})) implies that
    \begin{equation*}
        P[\CP_{\lambda(n)r_n^d}(B_y) = k_n-1] \leq \frac{(\lambda(n)\theta r_n^d)^{k_n-1}}{(k_n-1)!} e^{-n\theta r_n^d} \leq 2\beta \frac{k_n}{\lambda(n) (\lambda(n) \theta r_n^d)},
    \end{equation*}
    and thus, multiplying this to the above bound on the conditional probability, it holds that
    \begin{equation}
        h_n(x,y) \leq c_0 \frac{k_n}{\lambda(n) (\lambda(n) \theta r_n^d)} \left(\exp\left(-k_n F_{x \setminus y}^n H\left(\frac{1}{2}\right)\right) + \exp\left(- \frac{1}{2} k_n F_{y \setminus x}^n\right)\right).\label{estimate.01}
    \end{equation}

    Similarly, if $y \notin B_x$, then it follows that
    \begin{align*}
        &P[\CP_{\lambda(n)r_n^d}^y(B_x) = k_n \mid \CP_{\lambda(n)r_n^d}^x(B_y) = k_n]\\
        &= P[\CP_{\lambda(n)r_n^d}(B_x) = k_n \mid \CP_{\lambda(n)r_n^d}(B_y) = k_n]\\
        &\leq \exp\left(-k_n F_{x \setminus y}^n H\left(\frac{1}{2}\right)\right) + \exp\left(-\lambda(n)\theta r_n^d F_{x \setminus y}^n H\left(\frac{k_n F_{y \setminus x}^n}{2\lambda(n)\theta r_n^d F_{x \setminus y}^n}\right)\right)
    \end{align*}
    and the estimate as in (\ref{estimate.01}) holds in this case as well.
    Now, if $x \in [-1/(2r_n), 1/(2r_n)]^d$ is at distance at least 4 from the boundary, then $F_{x\setminus y}^n = F_{y\setminus x}^n = F_{x-y} := \text{Leb}(B_{x-y}\setminus B_{\bm{0}})/\theta$ holds for any $y \in B(x; 3)$.
    Therefore, we have
    \begin{align*}
        I_2(n) &\leq c_0 \lambda(n)^2r_n^{2d} \frac{k_n}{\lambda(n)(\lambda(n)\theta r_n^d)} \left(\int_{[-1/(2r_n), 1/(2r_n)]^d} \int_{B(x;3)} \exp(-c_1 k_n F_{x-y}) \,dydx + r_n^{1-d}\right)\\
        &\leq c_0 k_n r_n^d \left(r_n^{-d} \int_{B(\bm{0};3)} \exp(-c_1 k_n \|z\|_2) \,dz + r_n^{1-d}\right)\\
        &\leq c_0 (k_n^{1-d} + k_n r_n)\\
        &\to 0,
    \end{align*}
    where $\|\cdot\|_2$ denotes the Euclidean norm, and we used $r_n^d \ll \log n/n$, $k_n = o(\log n)$ and the fact that $\inf_{z \in B(\bm{0}; 3)} (F_z/\|z\|_2) > 0$ (see \cite[Proposition~5.16]{Penrose.2003}).
\end{proof}
We need to de-Poissonize the above result to go back to the binomial process.
\begin{prop}\label{Proposition.02}
    Suppose that $X_1$ is uniformly distributed on the $d$-dimensional unit cube.
    Then, for the sequence $(r_n)_{n \in \BN}$ taken in (\ref{condition.02}), $\sum_{j=k_n}^{n-1} W_{j,n} \xrightarrow{d} \text{Po}(\beta)$ holds.
\end{prop}
\begin{proof}
    Let $N_n^-$ and $M_n$ be independent Poisson random variables with mean $n-n^{3/4}$ and $2n^{3/4}$ respectively.
    Define Poisson point processes
    \begin{equation*}
        \CP_n^- := \sum_{i=1}^{N_n^-} \delta_{X_i}, \quad \CP_n^+ := \sum_{i=1}^{N_n^- + M_n} \delta_{X_i}
    \end{equation*}
    and let $W_{j,n}^-, W_{j,n}^+$ denote the number of vertices of degree $j$ in $G(\CP_n^-; r_n), G(\CP_n^+; r_n)$ respectively.
    Then, with high probability, $\CP_n^- \subset \CX_n \subset \CP_n^+$ holds and thus
    \begin{equation*}
        \sum_{j=k_n}^\infty W_{j,n}^- \leq \sum_{j=k_n}^{n-1} W_{j,n} \leq \sum_{j=k_n}^\infty W_{j,n}^+
    \end{equation*}
    holds.
    By Lemma \ref{Lemma.03},
    \begin{equation*}
        \sum_{j=k_n}^\infty W_{j,n}^- \to \text{Po}(\beta), \quad E\left[\sum_{j=k_n}^\infty W_{j,n}^+ - \sum_{j=k_n}^\infty W_{j,n}^-\right] \to 0
    \end{equation*}
    and by the constraction of $\CP_n^-, \CP_n^+$, always $\sum_{j=k_n}^\infty W_{j,n}^- \leq \sum_{j=k_n}^\infty W_{j,n}^+$.
    Combining these facts, we obtain $\sum_{j=k_n}^{n-1} W_{j,n} \xrightarrow{d} \text{Po}(\beta)$.
\end{proof}
One more lemma is required to complete the proof of Theorem \ref{MainResult.02}.
\begin{lem}\label{Lemma.04}
    Assume that $k_n = o(\log n)$ and the density $f$ of $X_1$ is bounded.
    Then, with high probability, it holds that $n\theta S_{k_n}(\CX_n)^d \leq o(k_n)$.
\end{lem}
\begin{proof}
    This follows immediately by limit in probability for maximum degree in subconnective regime (Penrose\cite[Theorem 6.10]{Penrose.2003}).
\end{proof}
\begin{proof}[Proof of Theorem \ref{MainResult.02}]
    Let $x \in \BR$ and $(r_n)_{n \in \BN}$ be a sequence satisfying the relation (\ref{condition.02}) for $\beta := e^{-x}$.
    Then, by Proposition \ref{Proposition.02}, the probability $P[S_{k_n}(\CX_n) > r_n]$ converges to $\exp(-\exp(-x))$ as $n \to \infty$.
    It follows from the choice of $(r_n)_{n \in \BN}$ in (\ref{condition.02}) that
    \begin{equation*}
        P[S_{k_n}(\CX_n) > r_n] = P\left[-\frac{n\theta S_{k_n}(\CX_n)^d}{k_n} < W_0\left(-\frac{1}{e} \left(\frac{n}{\beta \sqrt{2\pi k_n}}\right)^{-1/k_n}\right)\right].
    \end{equation*}
    By Lemma \ref{Lemma.04}, $-n\theta S_{k_n}(\CX_n)^d/k_n$ lies in the range of $W_0$ with high probability.
    Since $W_0$ is the inverse function of $te^t ~ (t \geq -1)$ which is strictly increasing, the above equals
    \begin{equation*}
        P\left[-\frac{n\theta S_{k_n}(\CX_n)^d}{k_n} \exp\left(-\frac{n\theta S_{k_n}(\CX_n)^d}{k_n}\right) < -\frac{n^{-1/k_n}}{e} (\beta \sqrt{2\pi k_n})^{1/k_n}\right] + o(1).
    \end{equation*}
    Note that
    \begin{equation}
        \beta = \exp(-(x-\log \sqrt{2\pi k_n}))/\sqrt{2\pi k_n} = \frac{1}{\sqrt{2\pi k_n}} \left(1-\frac{x-\log \sqrt{2\pi k_n}+\varepsilon_n}{k_n}\right)^{k_n} \label{beta}
    \end{equation}
    for some $\varepsilon_n \to 0$ since $(\log \sqrt{2\pi k_n})^2/k_n \to 0$.
    Therefore, the probability $P[S_{k_n}(\CX_n) > r_n]$ equals
    \begin{equation*}
        \begin{aligned}
            &P\left[-\frac{n\theta S_{k_n}(\CX_n)^d}{k_n} \exp\left(-\frac{n\theta S_{k_n}(\CX_n)^d}{k_n}\right) < -\frac{n^{-1/k_n}}{e} \left(1-\frac{x-\log \sqrt{2\pi k_n} + \varepsilon_n}{k_n}\right)\right] + o(1)\\
            &= P\left[-en ^{1+1/k_n}\theta S_{k_n}(\CX_n)^d \exp\left(-\frac{n\theta S_{k_n}(\CX_n)^d}{k_n}\right) + k_n + \log \sqrt{2\pi k_n} + \varepsilon_n < x\right] + o(1).
        \end{aligned}
    \end{equation*}
    In the above, $\varepsilon_n$ is deterministic and thus it converges to $0$ in probability.
    Therefore, it holds that
    \begin{equation*}
        P\left[-en ^{1+1/k_n}\theta S_{k_n}(\CX_n)^d \exp\left(-\frac{n\theta S_{k_n}(\CX_n)^d}{k_n}\right) + k_n + \log \sqrt{2\pi k_n} < x\right] \xrightarrow{n \to \infty} \exp(-\exp(-x)).
    \end{equation*}
    by Slutsky's theorem.
\end{proof}
\begin{rem}\label{Remark.02}
    To prove the assertion made in Remark~\ref{Remark.01}, note that
    \begin{equation*}
        \frac{n^{1+1/k_n}\theta r_n^d}{k_n} = -n^{1/k_n} W_0 \left(-\frac{1}{e} \left(\frac{n}{\beta \sqrt{2\pi k_n}}\right)^{-1/k_n}\right) = \frac{1}{e} (\beta \sqrt{2\pi k_n})^{1/k_n} + O(n^{-1/k_n})
    \end{equation*}
    from the Taylor expansion of $W_0$.
    Then, the equation (\ref{beta}) and the assumption $k_n n^{-1/k_n} \to 0$ imply
    \begin{equation*}
        -e n^{1+1/k_n} \theta r_n^d = - k_n + (x - \log \sqrt{2\pi k_n} + o(1)) + o(1) = x - (k_n + \log \sqrt{2\pi k_n} + \varepsilon_n')
    \end{equation*}
    for some $\varepsilon_n' \to 0$.
    From Lemma~\ref{Lemma.04}, it holds that
    \begin{equation*}
        \begin{aligned}
            &P[-en^{1+1/k_n} \theta S_{k_n}(\CX_n)^d + k_n + \log \sqrt{2\pi k_n} + \varepsilon_n' < x]\\
            &= P[-en^{1+1/k_n} \theta S_{k_n}(\CX_n)^d < -en^{1+1/k_n} \theta r_n^d]\\
            &= P[S_{k_n}(\CX_n) > r_n]\\
            &\xrightarrow{n \to \infty} \exp(-\exp(-x)).
        \end{aligned}
    \end{equation*}
    This, together with Slutsky's theorem, establishes the assertion.
\end{rem}

From now on, let us consider the case where the density $f$ of $X_1$ is as in Theorem \ref{MainResult.03}.
As in the previous case, we will find a sequence $(r_n)_{n \in \BN}$ satisfying
\begin{equation*}
    n\theta r_n^d = o(k_n), \quad E[W_{k_n,n}'] \to \beta
\end{equation*}
for some $\beta \in (0,\infty)$.
In this case,
\begin{align*}
    E[W_{k_n,n}'] &= n\int_{\BR^d} \frac{(nF(B(x;r_n)))^{k_n}}{k_n!} e^{-nF(B(x;r_n))} f(x)\,dx\\
    &= n\frac{(n\theta r_n^d f_{\max})^{k_n}}{k_n!} e^{-n\theta r_n^d f_{\max}} \int_{\BR^d} \left(\frac{F(B(x;r_n))}{\theta r_n^d f_{\max}}\right)^{k_n} \exp(n\theta r_n^d f_{\max} - nF(B(x;r_n))) f(x)\,dx.
\end{align*}
Let us now estimate the integral term.
Let $\rho_0 > \rho_1 > 0$ and $\varepsilon_0 > 0$ be as in the conditions of $f$, then
\begin{align*}
    &\int_{\BR^d \setminus B(\bm{0}; \rho_1+r_n)} \left(\frac{F(B(x;r_n))}{\theta r_n^d f_{\max}}\right)^{k_n} \exp(n\theta r_n^d f_{\max} - nF(B(x;r_n))) f(x)\,dx\\
    &\leq \int_{\BR^d \setminus B(\bm{0}; \rho_1+r_n)} (1-\varepsilon_0)^{k_n} \exp(\varepsilon_0 n\theta r_n^d f_{\max}) f(x)\,dx\\
    &\leq \exp(-ck_n)
\end{align*}
holds since $t^{k_n} e^{-t}$ is increasing on $[0,k_n]$ and $n\theta r_n^d = o(k_n)$.
On the other hand, since $\rho_1+r_n \leq (\rho_0 + \rho_1)/2$ for all large enough $n$ and $f$ is Lipschitz on $B(\bm{0}; (\rho_0 + \rho_1)/2)$, it holds that
\begin{align*}
    &\int_{B(\bm{0}; \rho_1+r_n)} \left(\frac{F(B(x;r_n))}{\theta r_n^d f_{\max}}\right)^{k_n} \exp(n\theta r_n^d f_{\max} - nF(B(x;r_n))) f(x)\,dx\\
    &= \int_{B(\bm{0}; \rho_1+r_n)} \left(\frac{f(x)}{f_{\max}}\right)^{k_n}(1+O(r_n))^{k_n} \exp(n\theta r_n^d f_{\max} - n\theta r_n^d f(x)(1+O(r_n))) f(x)\,dx\\
    &\sim \int_{B(\bm{0}; \rho_1+r_n)} \left(\frac{f(x)}{f_{\max}}\right)^{k_n} \exp(n\theta r_n^d f_{\max} - n\theta r_n^d f(x)) f(x)\,dx,
\end{align*}
where we used the facts $k_n r_n \to 0$ and $n\theta r_n^d \times r_n \to 0$ which follows from the assumptions $n\theta r_n^d = o(k_n)$ and $k_n = o(\log n)$.
By the Taylor expansion, it equals
\begin{equation*}
    \int_{B(\bm{0}; \rho_1+r_n)} \left(1+\frac{\partial^{2m}f(\xi_x x)[x]}{(2m)! f_{\max}}\right)^{k_n} \exp(-n\theta r_n^d \partial^{2m}f(\xi_x x)[x]/((2m)!)) f(x)\,dx
\end{equation*}
where $\xi_x \in (0,1)$.
By the change of variables $x = k_n^{-1/(2m)} z$, the above equals to
\begin{equation}\label{equation.01}
    \begin{aligned}
        &k_n^{-d/(2m)} \int_{B(\bm{0}; k_n^{1/(2m)}(\rho_1+r_n))} \left(1+\frac{\partial^{2m}f(\xi_z' k_n^{-1/(2m)} z)[z]}{(2m)! f_{\max} k_n}\right)^{k_n}\\
        &\quad \times \exp\left(-\frac{n\theta r_n^d}{(2m)! k_n} \partial^{2m}f(\xi_z' k_n^{-1/(2m)} z)[z]\right) f(k_n^{-1/(2m)} z)\,dz,
    \end{aligned}
\end{equation}
where $\xi_z' := \xi_{k_n^{-1/(2m)}z}$.
For all large enough $n$, the integrand is dominated by
\begin{equation*}
    f_{\max} \exp(\partial^{2m}f(\xi_z' k_n^{-1/(2m)} z)[z]/(2 ((2m)!) f_{\max}))
\end{equation*}
since $n\theta r_n^d = o(k_n)$, and asymptotic to
\begin{equation*}
    \exp(\partial^{2m}f(\bm{0})[z]/((2m)! f_{\max})) f_{\max}
\end{equation*}
since $f$ is $C^{2m}$ on some neighborhood of $\bm{0}$.
By continuity of the derivatives of $f$, for some $c > 0$,
\begin{equation*}
    \partial^{2m}f(\xi_z' k_n^{-1/(2m)} z)[z]/(2 ((2m)!) f_{\max}) \leq -c\|z\|_2^{2m}
\end{equation*}
holds for all large enough $n$.
Thus, (\ref{equation.01}) is asymptotic to
\begin{align*}
    &k_n^{-d/(2m)} \int_{\BR^d} \exp(\partial^{2m}f(\bm{0})[z]/((2m)! f_{\max})) f_{\max}\,dz\\
    &= k_n^{-d/(2m)} \int_{\BR^d} \exp(\partial^{2m}f(\bm{0})[z]/((2m)!)) \,dz \times f_{\max}^{1+d/(2m)}
\end{align*}
by the dominated convergence theorem.
Combining all the above, we obtain
\begin{equation*}
    E[W_{k_n,n}'] \sim \gamma f_{\max}^{1+d/(2m)} n\frac{(n\theta r_n^d f_{\max})^{k_n}}{k_n!k_n^{d/(2m)}} e^{-n\theta r_n^d f_{\max}},
\end{equation*}
where $\gamma := \int_{\BR^d} \exp(\partial^{2m}f(\bm{0})[z]/((2m)!)) \,dz$.
To ensure that this converges to $\beta$, by Stirling's formula, it suffices to choose
\begin{equation*}
    \frac{n\theta r_n^d f_{\max}}{k_n} = -W_0\left(-\frac{1}{e}\left(\frac{n \gamma f_{\max}^{1+d/(2m)}}{\beta \sqrt{2\pi k_n^{1+(d/m)}}}\right)^{-1/k_n}\right)
\end{equation*}
as in the uniform case.
The rest of the proof is quite similar to that of Theorem~\ref{MainResult.02}.
We describe only the necessary modifications to Lemma 5.2.
We split the integral $I_2(n)$ into two parts:
\begin{equation*}
    \begin{aligned}
        I_2(n) &= \lambda(n)^2 \int_{B(\bm{0}; \rho_1 + 3r_n)} F(dx) \int_{B(x; 3r_n)} F(dy) P[\{\CP_{\lambda(n)}^y(B_x^n) = k_n\} \cap \{\CP_{\lambda(n)}^x(B_y^n) = k_n\}]\\
        &\quad + \lambda(n)^2 \int_{\BR^d \setminus B(\bm{0}; \rho_1 + 3r_n)} F(dx) \int_{B(x; 3r_n)} F(dy) P[\{\CP_{\lambda(n)}^y(B_x^n) = k_n\} \cap \{\CP_{\lambda(n)}^x(B_y^n) = k_n\}].
    \end{aligned}
\end{equation*}
Note that we can take $\rho_0$ such that $f$ is bounded away from $0$ on $B(\bm{0}; \rho_0)$.
Since $f(x) \leq (1-\varepsilon) f_{\max}$ on $B(\bm{0};\rho_1)$, the second term in the above is bounded by
\begin{equation*}
    \lambda(n)^2 r_n^d \frac{k_n}{\lambda(n) \theta r_n^d} \frac{(\lambda(n) \theta r_n^d f_{\max})^{k_n}}{k_n!} e^{-n\theta r_n^d f_{\max}} (1-\varepsilon_0/2)^{k_n} \leq k_n^{1+d/(2m)} \exp(-ck_n)
\end{equation*}
and this vanishes as $n \to \infty$.
We can verify that the first term tends to $0$ in the same way as in the proof of Lemma \ref{Lemma.03}.
On the other hand, it holds that
\begin{equation*}
    I_1(n) \leq \left(\lambda(n)\frac{(\lambda(n)\theta r_n^d f_{\max})^{k_n}}{k_n!}\right)^2 r_n^d \leq ck_n^{2+(d/m)} r_n^d \to 0,
\end{equation*}
since $r_n$ decays faster than any power of $k_n^{-1}$.
\begin{proof}[Proof of Corollary~\ref{Corollary.01}]
    Let $\varepsilon > 0$.
    Then, in the uniform case,
    \begin{align*}
        &P\left[\left|\frac{n^{1+1/k_n}\theta S_{k_n}(\CX_n)^d}{k_n} - \frac{1}{e}\right| \geq \varepsilon\right]\\
        &\leq P\left[\left|-en^{1+1/k_n}\theta S_{k_n}(\CX_n)^d \exp\left(-\frac{n\theta S_{k_n}(\CX_n)^d}{k_n}\right) + k_n + \log \sqrt{2\pi k_n}\right| \geq k_n e\varepsilon/2\right] + o(1)\\
        &\leq P\left[\left|-en^{1+1/k_n}\theta S_{k_n}(\CX_n)^d \exp\left(-\frac{n\theta S_{k_n}(\CX_n)^d}{k_n}\right) + k_n + \log \sqrt{2\pi k_n}\right| \geq M\right] + o(1)
    \end{align*}
    holds for all sufficiently large $n$ and all $0 < M \leq k_n e\varepsilon/2$.
    Thus, by Theorem 2.2 and the Portmanteau theorem, we have
    \begin{equation*}
        \limsup_{n \to \infty} P\left[\left|\frac{n^{1+1/k_n}\theta S_{k_n}(\CX_n)^d}{k_n} - \frac{1}{e}\right| \geq \varepsilon\right] \leq P[|Z| \geq M].
    \end{equation*}
    Since this holds for any $M > 0$, the left-hand side must be zero.
    The proof for the non-uniform case is almost identical.
\end{proof}

\subsection{Point configuration}
Unlike the case where $k_n$ is fixed, we first prove the result for the Poissonized version and then approximate the binomial version.
The following is a refinement of Lemma \ref{Lemma.02}, adapted for proving the convergence of point processes.
\begin{prop}\label{Proposition.05}
    Let $(A_n)_{n \in \BN}$ be a sequence of subsets of $\BZ_{\geq 0}$, and $(r_n)_{n \in \BN}, (\lambda(n))_{n \in \BN}$ be a sequence of positive numbers.
    For $C \in \CC$, let
    \begin{equation*}
        W_{A_n,\lambda(n)}'(C) = \sum_{Y \in \CP_{\lambda(n)}} \mathbf{1}\{\CP_{\lambda(n)}(B(Y;r_n) \setminus \{Y\}) \in A_n\} \cap \{Y \in C\}.
    \end{equation*}
    Define
    \begin{equation*}
        \begin{aligned}
            &I_{1,C}(n) := \lambda(n)^2 \int_C F(dx) \int_{B(x;3r_n) \cap C} F(dy) P[\CP_{\lambda(n)}(B_x^n) \in A_n] P[\CP_{\lambda(n)}(B_y^n) \in A_n],\\
            &I_{2,C}(n) := \lambda(n)^2 \int_C F(dx) \int_{B(x;3r_n) \cap C} F(dy) P[\{\CP_{\lambda(n)}^y(B_x^n) \in A_n\} \cap \{\CP_{\lambda(n)}^x(B_y^n) \in A_n\}],
        \end{aligned}
    \end{equation*}
    where $B_x^n := B(x;r_n)$ and $\CP_{\lambda(n)}^x := \CP_{\lambda(n)} + \delta_x$.
    Then,
    \begin{equation*}
        d_\text{TV}(W_{A_n,\lambda(n)}'(C), \text{Po}(E[W_{A_n,\lambda(n)}'(C)])) \leq \min\left(3, \frac{1}{E[W_{A_n,\lambda(n)}'(C)]}\right) (I_{1,C}(n) + I_{2,C}(n))
    \end{equation*}
    holds.
\end{prop}
\begin{proof}
    For a given $m \in \BN$, we partition $\BR^d$ into cubes with side length $m^{-1}$.
    Then, let $\{H_{m,i}\}_{i=1}^{\ell}$ denote the collection of all such cubes that intersect $C$, and let $\{a_{m,i}\}_{i=1}^{\ell}$ be their corresponding centers.
    The labels are assigned according to the lexicographic ordering of the centers.
    Since $C$ is bounded, these are finite sets.
    For each $m,i$, define
    \begin{equation*}
        \xi_{m,i} := \mathbf{1}\{\CP_{\lambda(n)}(H_{m,i} \cap C) = 1\} \cap \{\CP_{\lambda(n)}(B_{a_{m,i}}^n \setminus H_{m,i}) \in A_n\}
    \end{equation*}
    and, set $p_{m,i} := E[\xi_{m,i}]$ and $p_{m,i,j} := E[\xi_{m,i}\xi_{m,j}]$.

    Let $I_m := \{1, 2, \cdots, \ell\}$ and define an adjacency relation on $I_m$ by
    \begin{equation*}
        i \sim_m j \iff 0 < \|a_{m,i} - a_{m,j}\| \leq 3r_n.
    \end{equation*}
    Then, if $m$ is sufficiently large, $(I_m, \sim_m)$ is a dependency graph for $(\xi_{m,i})_{1 \leq i \leq \ell}$ since the Poisson point process has the spatial independence property.
    Let us define $\widetilde{W}_m := \sum_{i \in I_m} \xi_{m,i}$.
    Then, $\lim_{m \to \infty} \widetilde{W}_m = W_{A_n,\lambda(n)}'(C)$ and by Lemma \ref{Lemma.05}, it holds that
    \begin{equation*}
        d_{\text{TV}}(\widetilde{W}_m, \text{Po}(E[\widetilde{W}_m])) \leq \min \left(3,\frac{1}{E[\widetilde{W}_m]}\right)\left(\sum_{i \in I_m} \sum_{j \in \CN_{m,i} \setminus \{i\}} p_{m,i,j} + \sum_{i \in I_m} \sum_{j \in \CN_{m,i}} p_{m,i}p_{m,j}\right),
    \end{equation*}
    where $\CN_{m,i} := \{j \in I_m : j \sim i\} \cup \{i\}$.
    Let us define the function $w_m$ by $w_m(x) := m^d p_{m,i}$ if $x \in H_{m,i}$ and $w_m(x) := 0$ if $x \notin H_{m,i}$ for all $i \in I_m$.
    Since $f \mathbf{1}_C$ is integrable, 
    \begin{equation*}
        \begin{aligned}
            w_m(x) &= m^d P[\{\CP_{\lambda(n)}(H_{m,i} \cap C) = 1\} \cap \{\CP_{\lambda(n)}(B_{a_{m,i}}^n \setminus H_{m,i}) \in A_n\}]\\
            &= \lambda(n) m^d \int_{H_{m,i}} f(y) \mathbf{1}_C(y) \,dy \times e^{O(m^{-d})} P[\CP_{\lambda(n)}(B_{a_{m,i}}^n \setminus H_{m,i}) \in A_n]\\
            &\xrightarrow{m \to \infty} \lambda(n) f(x) \mathbf{1}_C(x) P[\CP_{\lambda(n)}(B_x^n) \in A_n], \quad \text{a.s.}
        \end{aligned}
    \end{equation*}
    holds.
    Also,
    \begin{equation*}
        w_m(x) \leq m^d P[\CP_{\lambda(n)}(H_{m,i}) = 1] \leq \lambda(n) f_{\max}
    \end{equation*}
    and $w_m(x) = 0$ outside of some bounded set.
    Thus, by the dominated convergence theorem and Palm theory, $E[\widetilde{W}_m] \xrightarrow{m \to \infty} \lambda(n) \int_C P[\CP_{\lambda(n)}(B_x^n) \in A_n] f(x) \,dx = E[W_{A_n,\lambda(n)}']$.

    Similarly, define $u_m, v_m$ by
    \begin{equation*}
        u_m(x,y) := m^{2d} p_{m,i} p_{m,j} \mathbf{1}\{j \in \CN_{m,i}\}, \quad v_m(x,y) := m^{2d} p_{m,i,j} \mathbf{1}\{j \in \CN_{m,i} \setminus \{i\}\}
    \end{equation*}
    if $x \in H_{m,i}, y \in H_{m,j}$.
    By similar arguments, it holds that
    \begin{equation*}
        \begin{aligned}
            &u_m(x,y) \to \lambda(n)^2 f(x)\mathbf{1}_C(x) f(y)\mathbf{1}_C(y) P[\CP_{\lambda(n)}(B_x^n) \in A_n]P[\CP_{\lambda(n)}(B_y^n) \in A_n] \mathbf{1}_{B(x,3r_n)}(y),\\
            &v_m(x,y) \to \lambda(n)^2 f(x)\mathbf{1}_C(x) f(y)\mathbf{1}_C(y) P[\{\CP_{\lambda(n)}^y(B_x^n) \in A_n\} \cap \{\CP_{\lambda(n)}^x(B_y^n) \in A_n\}] \mathbf{1}_{B(x,3r_n)}(y)
        \end{aligned}
    \end{equation*}
    almost everywhere and
    \begin{equation*}
        \sum_{i \in I_m} \sum_{j \in \CN_{m,i}} p_{m,i}p_{m,j} \to I_{1,C}(n), \quad \sum_{i \in I_m} \sum_{j \in \CN_{m,i} \setminus \{i\}} p_{m,i,j} \to I_{2,C}(n).
    \end{equation*}
    This completes the proof.
\end{proof}
\begin{proof}[Proof of Theorem \ref{MainResult.07}]
    Note that in the both cases, the limiting point process $\CQ$ is simple.
    By Kallenberg's result~\cite[Theorem~16.16]{Kallenberg.2002}, $\Phi_{k_n,n} \xrightarrow{d} \CQ$ if and only if for all $C \in \CC_{\CQ}$, $\Phi_{k_n,n}(C) \xrightarrow{d} \CQ(C)$.

    Firstly, let us consider the uniform case.
    Let $C \in \CC_{\CQ}$.
    By an argument similar to that in the proof of Lemma~\ref{Lemma.03}, we obtain $d_{\text{TV}}(W_{k_n,\lambda(n)}'(C), \text{Po}(E[W_{k_n,\lambda(n)}'(C)])) \to 0$.
    By the Palm theory,
    \begin{equation*}
        \begin{aligned}
            E[W_{k_n,\lambda(n)}'(C)] &= \lambda(n) \int_{C \cap [-1/2,1/2]^d} \frac{(\lambda(n)F(B(x;r_n)))^{k_n}}{k_n!} e^{-\lambda(n)F(B(x;r_n))} \,dx\\
            &\sim n\frac{(\lambda(n)\theta r_n^d)^{k_n}}{k_n!} e^{-\lambda(n)\theta r_n^d} \times \text{Leb}(C \cap [-1/2,1/2]^d)\\
            &\sim \beta \text{Leb}(C \cap [-1/2,1/2]^d)
        \end{aligned}
    \end{equation*}
    holds.
    Therefore, $W_{k_n,\lambda(n)}'(C) \xrightarrow{d} \text{Po}(\beta \text{Leb}(C \cap [-1/2,1/2]^d)) = \CQ(C)$ and as in the proof of Proposition~\ref{Proposition.02}, it also holds that $\Phi_{k_n,n}(C) \xrightarrow{d} \CQ(C)$.
    This completes the proof for the uniform case.

    Next, we turn our attention to the non-uniform case.
    Let $C \in \CC_{\CQ}$ and we consider $\Phi_{k_n-1,n}(k_n^{-1/(2m)}C)$.
    Then,
    \begin{equation*}
        \begin{aligned}
            &E[W_{k_n,\lambda(n)}'(k_n^{-1/(2m)}C)]\\
            &\sim n \int_{k_n^{-1/(2m)}C} \frac{(n\theta r_n^d f(x))^{k_n}}{k_n!} e^{-n\theta r_n^d f(x)} f(x) \,dx\\
            &\sim n \frac{(n\theta r_n^d f_{\max})^{k_n}}{k_n! k_n^{d/(2m)}} e^{-n\theta r_n^d f_{\max}} \int_C \left(1+\frac{\partial^{2m}f(\bm{0})[z]}{(2m)! f_{\max} k_n}\right)^{k_n} \exp(n\theta r_n^d (f(k_n^{-1/(2m)}z)-f_{\max})) f_{\max} \,dz\\
            &\sim \beta \gamma^{-1} \int_C \exp(\partial^{2m}f(\bm{0})[z]/((2m)!)) \,dz = \Lambda(C).
        \end{aligned}
    \end{equation*}
    Thus, it follows that $W_{k_n,n}(k_n^{-1/(2m)}C) \xrightarrow{d} \CQ(C)$ and $D_{k_n^{1/(2m)}} \Phi_{k_n,n}(C) \xrightarrow{d} \CQ(C)$ as in the threshold radius case. 
    This completes the proof for the non-uniform case.
\end{proof}
\begin{proof}[Proof of Theorem~\ref{MainResult.08}]
    In this case, the limiting point process is simple.
    We aim to show that for all $C \in \CC_{\CH}$, $D_{(k_n^{1+d/(2m)}/(n\theta r_n^d))^{1/d}} \Phi_{k_n-1,n}(C) \xrightarrow{d} \CH(C)$.
    Note that
    \begin{equation*}
        D_{(k_n^{1+d/(2m)}/(n\theta r_n^d))^{1/d}} \Phi_{k_n-1,n}(C) = \Phi_{k_n-1,n}((n\theta r_n^d/k_n^{1+d/(2m)})^{1/d}C).
    \end{equation*}
    We now consider $W_{k_n-1,\lambda(n)}'((n\theta r_n^d/k_n^{1+d/(2m)})^{1/d}C)$, then its expectation is asymptotic to
    \begin{equation*}
        \begin{aligned}
            &n \int_{(n\theta r_n^d/k_n^{1+d/(2m)})^{1/d}C} \frac{(n\theta r_n^d f(x))^{k_n-1}}{(k_n-1)!} e^{-n\theta r_n^d f(x)} f(x) \,dx\\
            &\sim n\frac{n\theta r_n^d}{k_n! k_n^{d/(2m)}} \int_C (n\theta r_n^d f((n\theta r_n^d/k_n^{1+d/(2m)})^{1/d}z))^{k_n-1} \exp(-n\theta r_n^d f((n\theta r_n^d/k_n^{1+d/(2m)})^{1/d}z)) f_{\max} \,dz\\
            &\sim \gamma f_{\max}^{1+d/(2m)} n\frac{(n\theta r_n^d f_{\max})^{k_n}}{k_n! k_n^{d/(2m)}} e^{-n\theta r_n^d f_{\max}} \times \text{Leb}(C)/f_{\max}\\
            &\sim \frac{\beta}{f_{\max}} \text{Leb}(C)
        \end{aligned}
    \end{equation*}
    since $f$ is $2m$-H\"{o}lder continuous at $\bm{0}$.
    In the uniform case, we defined $\gamma := 1$.
    In a similar manner to the case of the threshold radius, it can be shown that the Poisson approximation holds.
    Therefore, 
    \begin{equation*}
        W_{k_n-1,\lambda(n)}'((n\theta r_n^d/k_n^{1+d/(2m)})^{1/d}C) \xrightarrow{d} \text{Po}((\beta/f_{\max})\text{Leb}(C))
    \end{equation*}
    holds and by an argument similar to Proposition~\ref{Proposition.02}, it also follows that
    \begin{equation*}
        \Phi_{k_n-1,n}((n\theta r_n^d/k_n^{1+d/(2m)})^{1/d}C) \xrightarrow{d} \text{Po}((\beta/f_{\max})\text{Leb}(C)).
    \end{equation*}
    This completes the proof.
\end{proof}

\appendix
\section{Appendix: Tail estimates and an expectation formula}
We summarize several fundamental tail estimates and expectation formulas for functionals defined on point configurations that are frequently employed in the study of random geometric graphs.
For their proofs, we refer the reader to Penrose~\cite[Section~1]{Penrose.2003}.
We define the function $H : [0,\infty) \to [0,\infty)$ by
\begin{equation*}
    H(t) := 1-t+t\log t \quad (t \in (0,\infty))
\end{equation*}
and $H(0) = 1$ with the convention $0 \log 0 = 0$.
$H(1) = 0$ is the unique zero of $H$, and it holds that $H(t) \geq t$ for all $t \geq 7$.
\begin{lem}
    Assume that $n, k \in \BN$ and $p \in (0,1)$.
    If $k \geq np$, then
    \begin{equation*}
        P[\text{Bi}(n,p) \geq k] \leq \exp\left(-np H\left(\frac{k}{np}\right)\right)
    \end{equation*}
    holds and if $k \leq np$,
    \begin{equation*}
        P[\text{Bi}(n,p) \leq k] \leq \exp\left(-np H\left(\frac{k}{np}\right)\right)
    \end{equation*}
    holds.
\end{lem}
\begin{lem}
    Assume that $\lambda \in (0,\infty)$ and $k \in \BN$.
    Then, if $k \geq \lambda$,
    \begin{equation*}
        P[\text{Po}(\lambda) \geq k] \leq \exp\left(-\lambda H\left(\frac{k}{\lambda}\right)\right)
    \end{equation*}
    holds and if $k \leq \lambda$,
    \begin{equation*}
        P[\text{Po}(\lambda) \leq k] \leq \exp\left(-\lambda H\left(\frac{k}{\lambda}\right)\right)
    \end{equation*}
    holds.
\end{lem}
\begin{lem}
    For all $\lambda \in (0,\infty)$ large enough,
    \begin{equation*}
        P[\text{Po}(\lambda) \geq \lambda + \lambda^{3/4}/2] \leq \exp(-\lambda^{1/2}/9), \quad P[\text{Po}(\lambda) \leq \lambda - \lambda^{3/4}/2] \leq \exp(-\lambda^{1/2}/9)
    \end{equation*}
    hold.
\end{lem}

For measurable functions on the product measurable space $\mathbb{N}(\BR^d) \times \mathbb{N}(\BR^d)$, we have the following expectation formula, which is a fundamental result in the Palm theory for Poisson processes:
\begin{lem}
    Assume that $\lambda \in (0,\infty)$ and $j \in \BN$.
    Let $h : \mathbb{N}(\BR^d) \times \mathbb{N}(\BR^d) \to [0,\infty]$ be a measurable function such that for $\CY, \CX \in \mathbb{N}(\BR^d)$, $h(\CY,\CX) = 0$ unless $\CY (\BR^d) = j$.
    Then,
    \begin{equation*}
        E\left[\sum_{\CY \subset \CP_{\lambda}} h(\CY, \CP_{\lambda})\right] = \frac{\lambda^j}{j!} E[h(\CX_j', \CX_j' + \CP_{\lambda})]
    \end{equation*}
    holds, where $\CX_j'$ is an independent copy of $\CX_j$.
\end{lem}

\section*{Acknowledgement}
The author would like to thank Professor Ryoki Fukushima for his constant guidance and fruitful discussions.
He also provided many helpful comments through a careful reading of the manuscript.

\end{document}